\documentclass[english]{scrartcl}

\usepackage[T1]{fontenc}
\usepackage[utf8]{inputenc}

\usepackage{amsmath,amsfonts,amsthm,amssymb}
\usepackage[colorlinks,citecolor=blue]{hyperref}
\usepackage{doi,natbib}
\usepackage{cancel}
\usepackage{relsize}

\usepackage[shortlabels]{enumitem}

\usepackage{changes}

\usepackage[figure]{hypcap}
\usepackage{graphicx}
\usepackage{subcaption}

\usepackage{tikz}
\usetikzlibrary{arrows}

\tikzset{
    punkt/.style={
           rectangle,
           draw=white, very thick,
           text width=10.0em,
           minimum height=1.5em,
           text centered}
}

\usepackage{preamble}

\usepackage{marginnote}

\usepackage[capitalise,nameinlink]{cleveref}
\allowdisplaybreaks

\newcommand\norm[1]{\Vert#1\Vert}
\newcommand\abs[1]{\left\vert#1\right\vert}

\newcommand\N{\mathbb{N}}
\newcommand\R{\mathbb{R}}

\newcommand{\dist}{\operatorname{dist}}
\newcommand{\dom}{\operatorname{dom}}
\newcommand{\gph}{\operatorname{gph}}

\newcommand{\cl}{\operatorname{cl}}
\DeclareMathOperator*{\argmin}{\operatorname{argmin}}

\newcommand{\tto}{\rightrightarrows}

\newtheorem{theorem}{Theorem}[section]
\newtheorem{lemma}[theorem]{Lemma}
\newtheorem{proposition}[theorem]{Proposition}
\newtheorem{assumption}[theorem]{Assumption}
\newtheorem{corollary}[theorem]{Corollary}
\newtheorem{remark}[theorem]{Remark}
\newtheorem{definition}[theorem]{Definition}
\newtheorem{example}[theorem]{Example}

\crefname{figure}{Figure}{Figures}

\begin{document}

\title{
	R-regularity of set-valued mappings under the relaxed constant positive
	linear dependence constraint qualification with applications to
	parametric and bilevel optimization
	}
\author{%
	Patrick Mehlitz%
	\footnote{%
		Brandenburgische Technische Universit\"at Cottbus--Senftenberg,
		Institute of Mathematics,
		03046 Cottbus,
		Germany,
		\email{mehlitz@b-tu.de},
		\url{https://www.b-tu.de/fg-optimale-steuerung/team/dr-patrick-mehlitz},
		ORCID: 0000-0002-9355-850X%
		}
	\and
	Leonid I.\ Minchenko%
	\footnote{%
		Belarus State University of Informatics and Radioelectronics,
		6 P.\ Brovki Street,
		Minsk 220013,
		Belarus,
		\email{leonidm@insoftgroup.com},
		ORCID: 0000-0002-8773-2559
		}
	}

\publishers{}
\maketitle

\begin{abstract}
	The presence of Lipschitzian properties for solution mappings
	associated with nonlinear parametric optimization problems is 
	desirable in the context of stability analysis or bilevel optimization.
	An example of such a Lipschitzian property for set-valued mappings,
	whose graph is the solution set of a system of nonlinear inequalities and
	equations, is R-regularity.	
	Based on the so-called relaxed constant positive linear dependence constraint qualification,
	we provide a criterion ensuring the presence of the R-regularity property.
	In this regard, our analysis generalizes earlier results of that type which exploited
	the stronger Mangasarian--Fromovitz or constant rank constraint qualification.
	Afterwards, we apply our findings in order to derive new sufficient conditions which guarantee
	the presence of R-regularity for
	solution mappings in parametric optimization. Finally, our results are used to derive
	an existence criterion for solutions in pessimistic bilevel optimization and a
	sufficient condition for the presence of the so-called partial calmness property 
	in optimistic bilevel optimization.
\end{abstract}


\begin{msc}	
	49J53, 90C30, 90C31
\end{msc}

\section{Introduction}\label{sec:introduction}

Lipschitzian properties of implicitly given set-valued mappings are of essential importance
in order to study the stability of optimization problems,
see e.g.\ \cite{GfrererOutrata2016,LudererMinchenkoSatsura2002,Mordukhovich2006}
and the references therein. Particularly, such stability is desirable in the
context of bilevel optimization where a function has to be minimized over the
graph of a solution mapping associated with a given parametric optimization
problem, see \cite{Bard1998,Dempe2002,DempeKalashnikovPerezValdesKalashnykova2015}
or \cref{sec:bilevel_optimization} for details. Indeed, in order to infer
existence results, optimality conditions, or solution algorithms in bilevel
programming, one generally
has to assume the presence of certain properties of this solution map.
However, it is often not easy to verify such properties. In this paper,
we focus on the derivation of sufficient criteria for the presence of so-called
\emph{R-regularity} of set-valued mappings, see \cref{def:R_regularity}. 
This property, in turn, is beneficial in order to study Lipschitzian properties
of marginal (or optimal value) functions and solution mappings in parametric optimization,
see \cite{BednarczukMinchenkoRutkowski2019,LudererMinchenkoSatsura2002,MinchenkoStakhovski2011b},
and these features possess some extensions to bilevel optimization as well.

In this paper, we investigate set-valued mappings $\Gamma\colon\R^n\tto\R^m$ of the form
\begin{equation}\label{eq:definition_of_Gamma}
	\forall x\in\R^n\colon\quad
	\Gamma(x):=
	\left\{
		y\in\R^m\,\middle|\,
		\begin{aligned}
			h_i(x,y)&\,\leq \,0&&i\in I\\h_i(x,y)&\,=\,0&&i\in J
		\end{aligned}
	\right\}
\end{equation}
where $I:=\{1,\ldots,\ell\}$ and $J:=\{\ell+1,\ldots,p\}$ are index sets
and $h_1,\ldots,h_p\colon\R^n\times\R^m\to\overline{\R}$ are given functions.
Precise assumptions on the continuity and smoothness properties of $h_1,\ldots,h_p$
will be specified in the course of the paper. 
It is well known that the presence of R-regularity for mappings of this type
is guaranteed under validity
of the Mangasarian--Fromovitz constraint qualification, see
\cite{Borwein1986,LudererMinchenkoSatsura2002}.
More recently, this result has been extended to situations where relaxed versions of the constant rank
constraint qualification hold at the underlying reference points, see \cite{BednarczukMinchenkoRutkowski2019,MinchenkoStakhovski2011b}.
However, in some situations, these qualification conditions may turn out to be too
selective in order to guarantee applicability of the obtained results in order to
investigate the presence of R-regularity for solution mappings, see e.g.\ \cref{rem:failure_of_MFCQ}.
That is why we aim for a generalization of these findings in the presence of the
so-called \emph{relaxed constant positive linear dependence constraint qualification},
introduced in \cite{AndreaniHaeserSchuverdtSilva2012}, which is generally weaker than
the aforementioned qualification conditions. 
Our main results \cref{thm:R_regularity_and_RCPLD,thm:R_regularity_and_RCPLD_with_inner_semicontinuity}
depict that this is indeed possible. 
With these new sufficient conditions for the presence of R-regularity for the mapping $\Gamma$
at hand, we are in position to state new criteria ensuring local Lipschitz continuity of
the marginal function and R-regularity of the solution mapping associated with nonlinear
parametric optimization problems whose feasible region is modeled with the aid of $\Gamma$.
Afterwards, we use these findings in order to study the existence of so-called pessimistic
solutions as well as the presence of the celebrated \emph{partial calmness} property 
in bilevel optimization. The latter, introduced in \cite{YeZhu1995}, is one of the key
assumptions one generally postulates on the optimal value reformulation of an optimistic
bilevel optimization problem in order to infer necessary optimality conditions and
solution algorithms, see \cref{sec:bilevel_optimization} for details and suitable references.

The remaining parts of this manuscript are organized as follows:
In \cref{sec:notation_and_preliminaries}, we provide the fundamental notation exploited in
this paper. Furthermore, we recall some important constraint qualifications from nonlinear
programming as well as the underlying fundamentals of set-valued analysis.
\Cref{sec:sufficient_conditions_for_R_regularity} is dedicated to the study of the
relaxed constant positive linear dependence constraint qualification as a sufficient condition for
R-regularity of the mapping $\Gamma$. In \cref{sec:applications}, we investigate some
applications of our findings. First, we apply the obtained results to nonlinear
parametric optimization problems in order to state new sufficient conditions for the
local Lipschitz continuity of the associated optimal value function as well as R-regularity
of the associated solution mapping in \cref{sec:parametric_optimization}.
Afterwards, we employ these results in the context of bilevel optimization in order to
formulate criteria ensuring the existence of pessimistic solutions as well as the
presence of partial calmness in \cref{sec:bilevel_optimization}.
In \cref{sec:conclusions}, we close the paper with the aid of some final comments.

\section{Notation and preliminaries}\label{sec:notation_and_preliminaries}

In this paper, we mainly make use of standard notation.
The tools of set-valued analysis we exploit here can be found, e.g., in \cite{BankGuddatKlatteKummerTammer1983,Mordukhovich2006,RockafellarWets1998}.

\subsection{Basic notation}

Throughout the paper, we equip $\R^n$ with the Euclidean norm $\norm{\cdot}$. 
For some point $x\in\R^n$ and a scalar $\varepsilon>0$, we use
\[
	\mathbb U_\varepsilon(x):=\{y\in\R^n\,|\,\norm{y-x}<\varepsilon\},
	\qquad
	\mathbb B_\varepsilon(x):=\{y\in\R^n\,|\,\norm{y-x}\leq\varepsilon\}
\]
in order to denote the open and closed $\varepsilon$-ball around $x$, respectively.
For brevity, we make use of $\mathbb B:=\mathbb B_1(0)$.
For a nonempty and closed set $A\subset\R^n$, we use 
\[
	\dist(x,A):=\inf\{\norm{y-x}\,|\,y\in A\},
	\qquad
	\Pi(x,A):=\argmin\{\norm{y-x}\,|\,y\in A\}
\]
to denote the distance of $x$ to $A$ and the set of projections of $x$ onto $A$,
respectively.
It is well known that the distance function $\dist(\cdot,A)\colon\R^n\to\R$ 
is Lipschitz continuous with Lipschitz modulus $1$. 
Generally, we call a map $\phi\colon\R^n\to\R^m$ locally Lipschitz
continuous at $x$ w.r.t.\ $\Omega\subset\R^n$ whenever there are $\delta>0$ and $L>0$ such that
\[
	\forall y,y'\in\mathbb U_\delta(x)\cap \Omega\colon\quad
	\norm{\phi(y)-\phi(y')}\leq L\norm{y-y'}
\]
holds. Note that this notion is only reasonable in the situation $x\in\cl \Omega$.
For $\Omega:=\R^n$, we recover the classical definition of local Lipschitz continuity.

Let $I_1$ as well as $I_2$ be finite index sets and let $(a^i)_{i\in I_1}\subset\R^n$
as well as $(b^i)_{i\in I_2}\subset\R^n$ be two given families of vectors.
We call the pair of families $\bigl((a^i)_{i\in I_1},(b^i)_{i\in I_2})$
\emph{positive-linearly dependent} whenever there are scalars
$\alpha_i\geq 0$, $i\in I_1$, and $\beta_i$, $i\in I_2$, which are not all vanishing 
such that
\[
	\mathsmaller\sum\nolimits_{i\in I_1}\alpha_ia^i
	+
	\mathsmaller\sum\nolimits_{i\in I_2}\beta_ib^i
	=
	0.
\]
Otherwise, we refer to this pair of families as \emph{positive-linearly independent}.
A family of vectors $(a^i)_{i\in I_1}$ is called positive-linearly dependent (independent) whenever
the pair of families $\bigl((a^i)_{i\in I_1},\varnothing\bigr)$ is positive-linearly
dependent (independent).

The following lemma follows from \cite[Lemma~1]{AndreaniHaeserSchuverdtSilva2012}.
\begin{lemma}\label{lem:Caratheodory_lemma}
	Let $v^1,\ldots,v^{r+s}\in\R^n$ be given vectors such that the family $(v^i)_{i=1}^r$
	is linearly independent. Furthermore, let $z\in\R^n\setminus\{0\}$ be given as
	$z=\sum_{i=1}^{r+s}\alpha_iv^i$ for reals $\alpha_1,\ldots,\alpha_{r+s}$ satisfying
	$\alpha_{r+1},\ldots,\alpha_{r+s}>0$.
	Then there exist an index set $\mathcal I\subset\{r+1,\ldots,r+s\}$ and
	reals $\bar\alpha_i$, $i\in\{1,\ldots,r\}\cup\mathcal I$, satisfying $\bar\alpha_i>0$
	for all $i\in\mathcal I$, such that the family 
	$(v^i)_{i\in\{1,\ldots,r\}\cup\mathcal I}$ is linearly independent and
	\[
		z=\mathsmaller\sum\nolimits_{i\in\{1,\ldots,r\}\cup\mathcal I}\bar\alpha_iv^i.
	\]
\end{lemma}

\subsection{Constraint qualifications in nonlinear programming}\label{sec:CQs}

Supposing that $\Gamma$ models the feasible region of a given parametric
optimization problem, certain constraint qualifications need to be imposed
on the images of $\Gamma$ in order to ensure that the associated Karush--Kuhn--Tucker
conditions provide a necessary optimality condition.
In this regard, we postulate the following
assumption which may hold throughout the section.
\begin{assumption}\label{ass:differentiability_properties}
	Let us fix a reference parameter $\bar x\in\R^n$ and some
	point $\bar y\in\Gamma(\bar x)$. Furthermore, 
	let all the functions $h_1,\ldots,h_p$ be continuous as well as continuously
	differentiable w.r.t.\ $y$ in a neighborhood of $\{\bar x\}\times\Gamma(\bar x)$.
\end{assumption} 

Let us now introduce the qualification conditions of our interest. Therefore,
we will exploit the set of indices associated with inequality
constraints active at $(\bar x,\bar y)$ which is defined as stated below:
\[
	I(\bar x,\bar y):=\{i\in I\,|\,h_i(\bar x,\bar y)=0\}.
\]
\begin{definition}\label{def:constraint_qualifications}
	We say that
	\begin{enumerate}
		\item[(a)] the \emph{linear independence constraint qualification} (LICQ)
			holds at $(\bar x,\bar y)$ whenever the family 
			$\bigl(\nabla_yh_i(\bar x,\bar y)\bigr)_{i\in I(\bar x,\bar y)\cup J}$
			is linearly independent,
		\item[(b)] the \emph{Mangasarian--Fromovitz constraint qualification} (MFCQ)
			holds at $(\bar x,\bar y)$ whenever the pair of families
			\[
				\left(
					\bigl(\nabla_yh_i(\bar x,\bar y)\bigr)_{i\in I(\bar x,\bar y)},
				 	\bigl(\nabla_yh_i(\bar x,\bar y)\bigr)_{i\in J}
				 \right)
			\]
			is positive-linearly independent,
		\item[(c)] the \emph{relaxed constant rank constraint qualification} (RCRCQ)
			holds at $(\bar x,\bar y)$ (w.r.t.\ $\Omega\subset\R^n$) whenever
			there is a neighborhood $U$ of $(\bar x,\bar y)$ such that for
			each set $K\subset I(\bar x,\bar y)$, the family 
			$(\nabla_yh_i(x,y))_{i\in K\cup J}$
			has constant rank on $U$ (on $U\cap(\Omega\times\R^m)$),
		\item[(d)] the \emph{relaxed constant positive linear dependence constraint qualification}
			(RCPLD) holds at $(\bar x,\bar y)$ (w.r.t.\ $\Omega\subset\R^n$)
			whenever there is a neighborhood
			$U$ of $(\bar x,\bar y)$ and an index set $S\subset J$ such that the following
			conditions hold:
			\begin{enumerate}
				\item[(i)] $\{\nabla _yh_i(\bar x,\bar y)\,|\,i\in S\}$ is a basis of the span
					of $\{\nabla _yh_i(\bar x,\bar y)\,|\,i\in J\}$,
				\item[(ii)] the family $(\nabla_yh_i(x,y))_{i\in J}$
					has constant rank on $U$ (on $U\cap(\Omega\times\R^m$), and
				\item[(iii)] for each set $K\subset I(\bar x,\bar y)$ such that the pair
					of families
					\[
						\left(
							\bigl(\nabla_yh_i(\bar x,\bar y)\bigr)_{i\in K},
				 			\bigl(\nabla_yh_i(\bar x,\bar y)\bigr)_{i\in S}
				 		\right)
					\]
					is positive-linearly dependent, the family $(\nabla _yh_i(x,y))_{i\in K\cup S}$
					is linearly dependent for each point $(x,y)\in U$
					(for each point $(x,y)\in U\cap(\Omega\times\R^m)$).
			\end{enumerate}
	\end{enumerate}
\end{definition}
While LICQ and MFCQ is are well-known constraint qualifications, RCRCQ and RCPLD are
less popular. Let us mention that RCRCQ, which has been introduced in
\cite{MinchenkoStakhovski2011}, is a less restrictive constraint
qualification than the classical \emph{constant rank constraint qualification},
see \cite{Janin1984}. On the other hand, RCPLD dates back to
\cite{AndreaniHaeserSchuverdtSilva2012} and generalizes the classical
\emph{constant positive linear dependence constraint qualification}, see
\cite{AndreaniMartinezSchuverdt2005,QiWei2000}. Checking these references,
one can observe that both MFCQ and RCRCQ individually imply validity of
RCPLD. However, neither does MFCQ imply validity of RCRCQ nor vice versa.
Clearly, LICQ is stronger than MFCQ and RCRCQ.
Let us mention that RCPLD is stable in the sense that whenever it is valid
at some reference point, then it also holds in a neighborhood of this point.
In order to see this, one may adapt the proof of \cite[Theorem~4]{AndreaniHaeserSchuverdtSilva2012},
which is stated in the non-parametric setting, to the situation at hand. 
Finally, we would like
to mention that the notion of RCPLD can be extended to non-smooth constraint
systems as well as complementarity-type feasible regions, and, thus, applies to
mathematical programs with complementarity constraints and 
different reformulations of bilevel optimization problems, 
see \cite{ChieuLee2013,GuoLin2013,XuYe2019} for details.

\subsection{Properties of set-valued mappings}\label{sec:set_valued_analysis}

Let $\Upsilon\colon\R^n\tto\R^m$ be a set-valued mapping.
We refer to the sets 
\[
	\gph\Upsilon:=\{(x,y)\in\R^n\times\R^m\,|\,y\in\Upsilon(x)\},
	\qquad
	\dom\Upsilon:=\{x\in\R^n\,|\,\Upsilon(x)\neq\varnothing\}
\] 
as graph and domain of $\Upsilon$, respectively.
Let us fix a point $\bar x\in\dom\Upsilon$.
We call $\Upsilon$ \emph{locally bounded} at $\bar x$ whenever there are a bounded set
$B\subset\R^m$ and a neighborhood $U\subset\R^n$ of $\bar x$ such that
$\Upsilon(x)\subset B$ holds for all $x\in U$.
One calls $\Upsilon$ \emph{upper semicontinuous} at $\bar x$ whenever 
for each open set $O\subset\R^m$ which satisfies $\Upsilon(\bar x)\subset O$,
there exists a neighborhood $U\subset\R^n$ of $\bar x$ such that
$\Upsilon(x)\subset O$ holds for all $x\in U$.
Recall that $\Upsilon$ is called \emph{lower semicontinuous}
at $\bar x$ (w.r.t.\ $\Omega\subset\R^n$) whenever for each
open set $O\subset\R^m$ with $\Upsilon(\bar x)\cap O\neq\varnothing$, there is
a neighborhood $U\subset\R^n$ of $\bar x$ such that $\Upsilon(x)\cap O\neq
\varnothing$ holds for all $x\in U$ (for all $x\in U\cap\Omega$). 
We call $\Upsilon$ \emph{inner semicontinuous} at some point $(\bar x,\bar y)\in\gph\Upsilon$
(w.r.t.\ $\Omega$) whenever for each sequence $\{x^k\}_{k\in\N}\subset\R^n$ 
($\{x^k\}_{k\in\N}\subset\Omega$)
converging to $\bar x$, there exists a sequence $\{y^k\}_{k\in\N}\subset\R^m$ which converges
to $\bar y$ and satisfies $y^k\in\Upsilon(x^k)$ for sufficiently large $k\in\N$. 
Note that $\Upsilon$ is lower semicontinuous at $\bar x$ (w.r.t.\ $\Omega$) if and only if
it is inner semicontinuous at each point from $\{\bar x\}\times\Upsilon(\bar x)$ (w.r.t.\ $\Omega$).
The situation $\Omega:=\dom\Upsilon$ will be of particular interest in this manuscript.

In the theory of set-valued analysis, there exist several different notions of
Lipschitzianity. Recall that $\Upsilon$ possesses the \emph{Aubin property} at
some point $(\bar x,\bar y)\in\gph\Upsilon$ (w.r.t.\ $\Omega$) whenever there
exist neighborhoods $U$ and $V$ of $\bar x$ and $\bar y$, respectively, as
well as a constant $\kappa>0$ such that
\[
	\forall x,x'\in U\;(\forall x,x'\in U\cap\Omega)\colon\quad
	\Upsilon(x)\cap V\subset\Upsilon(x')+\kappa\,\norm{x-x'}\mathbb B
\]
holds. 
One can easily check that whenever $\Upsilon$ possesses the Aubin property at 
$(\bar x,\bar y)$ (w.r.t.\ $\Omega$), then it is inner semicontinuous (w.r.t.\ $\Omega$)
at this point.
Using the concept of coderivatives which is based on the limiting normal cone
from variational analysis, one can formulate a necessary and sufficient condition
for the presence of the Aubin property for set-valued mappings with closed
graphs, see \cite[Theorem~4.10]{Mordukhovich2006}.
In \cite[Corollary~4.39]{Mordukhovich2006}, one can find a characterization 
of the Aubin property of $\Gamma$ from \eqref{eq:definition_of_Gamma} 
at some point of its graph under validity
of an MFCQ-type assumption. Let us, however, note that MFCQ from \cref{def:constraint_qualifications}
is only sufficient but not necessary for the presence of the Aubin property.
A recent study on the presence of the Aubin property for implicitly defined 
set-valued mappings of more general form can be found in \cite{GfrererOutrata2016}.

Let us now focus on the particular mapping $\Gamma$ from \eqref{eq:definition_of_Gamma}
in more detail. In this manuscript, we are interested in the property of $\Gamma$ being
so-called R-regular at a point of its graph, see \cite[Section~6.2]{LudererMinchenkoSatsura2002}.
\begin{definition}\label{def:R_regularity}
	Fix $(\bar x,\bar y)\in\gph\Gamma$ and some set $\Omega\subset\R^n$.
	Then $\Gamma$ is called \emph{R-regular} at $(\bar x,\bar y)$ (w.r.t.\ $\Omega$)
	whenever there exist a constant $\kappa>0$ and a neighborhood $U$ of
	$(\bar x,\bar y)$ such that the condition
	\begin{equation}\label{eq:R_regularity}
		\begin{aligned}
			&\forall  (x,y)\in U\;(\forall(x,y)\in U\cap(\Omega\times\R^m))\colon\\
			&\qquad\dist(y,\Gamma(x))
				\leq
				\kappa\,
				\max\bigl\{	0,
					\max\{h_i(x,y)\,|\,i\in I\},
					\max\{|h_i(x,y)|\,|\,i\in J\}
					\bigr\}
		\end{aligned}
	\end{equation}
	holds. 
\end{definition}
The notion of R-regularity can be traced back to
\cite{Fedorov1979,Ioffe1979} where it has been exploited as a constraint qualification.
Following \cite{BoschJouraniHenrion2004,FabianHenrionKrugerOutrata2010,Robinson1976}, 
one might be tempted to say that the presence of R-regularity is equivalent to
the validity of a local error bound condition at some reference point of the constraint system 
induced by $\Gamma$ provided the latter does not depend on the parameter. 
In this regard,
R-regularity of a \emph{parametric} constraint system is a generalization of the concept of error bounds.
We refer the interested reader to \cite{Ye1998} where the concept of so-called \emph{uniform parametric error bounds}, which is closely related to R-regularity, is studied.
Let us note that due to \cite[Theorem~1]{Robinson1976} or 
\cite[Theorem~3.2]{Borwein1986}, R-regularity of $\Gamma$ at
a given reference point is implied by validity of MFCQ at the latter.
A generalization of this result to the setting where the functions $h_1,\ldots,h_p$ are non-smooth
can be found in \cite{Yen1997}.
We would like to point out that R-regularity can be interpreted as a variant of \emph{metric
regularity}, see \cite{Ioffe2000} and the references therein, 
and is stronger than
\emph{metric subregularity} of the feasibility mapping associated with the given parametric constraint
set where the parameter is fixed, see \cite[Section~1]{GfrererMordukhovich2017}.
Furthermore, following \cite{GfrererMordukhovich2017,Robinson1976}, it is possible to generalize
the concept of R-regularity, which is called \emph{stability} or \emph{Robinson stability} in these papers,
to geometric constraint systems of the type
\[
	\tilde h(x,y)\in C
\]
where $\tilde h\colon\R^n\times\R^m\to\R^p$ is continuously differentiable w.r.t.\ $y$ and
$C\subset\R^p$ is a closed set.

Invoking \cite[Theorem~5.1]{BednarczukMinchenkoRutkowski2019},
one can easily check that whenever $\Gamma$ is R-regular at $(\bar x,\bar y)$
w.r.t.\ $\Omega$ while all the functions $h_1,\ldots,h_p$ are locally Lipschitz
continuous at this point, then $\Gamma$ possesses the Aubin property w.r.t.\
$\Omega$ at this point. 
By means of simple examples, one can check that the converse statement
does not hold in general even if the data functions are continuously
differentiable and, thus, locally Lipschitzian, 
see \cite[Example~1]{MinchenkoStakhovski2011b}.
The following result even holds in the absence of local Lipschitz continuity of the data functions.
\begin{lemma}\label{lem:inner_semicontinuity_via_R_regularity}
	Let $\Gamma$ be R-regular at some point $(\bar x,\bar y)\in\gph\Gamma$
	w.r.t.\ $\dom\Gamma$.
	Furthermore, let the functions $h_1,\ldots,h_p$ be continuous
	at $(\bar x,\bar y)$ and let
	$h_1(x,\cdot),\ldots,h_p(x,\cdot)\colon\R^m\to\R$ be continuous for
	each $x\in\dom\Gamma$ which comes from a neighborhood of $\bar x$.
	Then $\Gamma$ is inner semicontinuous at
	$(\bar x,\bar y)$ w.r.t.\ $\dom\Gamma$.
\end{lemma}
\begin{proof}
	The assumptions of the lemma particularly imply the existence of
	a constant $\kappa>0$ and some $\delta>0$ such that
	\[
		\dist(\bar y,\Gamma(x))
		\leq
		\kappa\,
			\max\bigl\{	0,
						\max\{h_i(x,\bar y)\,|\,i\in I\},
						\max\{|h_i(x,\bar y)|\,|\,i\in J\}
				\bigr\}
	\]
	holds for all $x\in \mathbb U_\delta(\bar x)\cap\dom\Gamma$.
	Thus, for each sequence $\{x^k\}_{k\in\N}\subset\dom\Gamma$
	with $x^k\to\bar x$, the estimate
	\[
		\Vert\bar y-y^k\Vert
		\leq
		\kappa\,
			\max\bigl\{	0,
						\max\{h_i(x^k,\bar y)\,|\,i\in I\},
						\max\{|h_i(x^k,\bar y)|\,|\,i\in J\}
				\bigr\}
	\]
	holds for sufficiently large $k\in\N$ where $y^k\in\Pi(\bar y,\Gamma(x^k))$ 
	is arbitrarily chosen. Note that $\Pi(\bar y,\Gamma(x^k))$ is nonempty
	for each $k\in\N$
	since $\Gamma(x^k)$ is nonempty and closed
	by continuity of $h_1(x^k,\cdot),\ldots,h_p(x^k,\cdot)$
	and the choice $x^k\in\dom\Gamma$
	for sufficiently large $k\in\N$.
	Exploiting the continuity of $h_1,\ldots,h_p$ at $(\bar x,\bar y)$,
	we find $\Vert\bar y-y^k\Vert\to 0$ as $k\to\infty$, i.e., $\Gamma$ is inner
	semicontinuous at $(\bar x,\bar y)$ w.r.t.\ $\dom\Gamma$.
\end{proof}

By definition, R-regularity of a set-valued mapping at a given reference point
is stable in the sense that it extends to points in a sufficiently small
neighborhood. However, we get the following even stronger stability property 
from \cite[Lemma~6.19]{LudererMinchenkoSatsura2002} which shows that the
modulus of R-regularity is uniformly bounded in a neighborhood of
a compact set of points where a given set-valued mapping is R-regular.
\begin{lemma}\label{lem:stability_of_R_regularity}
	Let $C\subset\gph\Gamma$ be compact and assume that $\Gamma$ is R-regular
	w.r.t.\ $\dom\Gamma$ at each point from $C$. 
	Furthermore, let $O$ be a neighborhood of $C$ where $h_1,\ldots,h_p$ are
	continuous.
	Then there exist a constant
	$\kappa>0$ and an open set $U$ such that $C\subset U\subset O$
	while \eqref{eq:R_regularity}
	holds with $\Omega:=\dom\Gamma$, i.e., there is a uniform modulus $\kappa$ 
	of R-regularity on $C$.
\end{lemma}

\section{A sufficient condition for R-regularity}\label{sec:sufficient_conditions_for_R_regularity}

If not stated otherwise, we assume that \cref{ass:differentiability_properties} holds
throughout the section. 
For simplicity, let us postulate that the functions $h_i(x,\cdot)\colon\R^m\to\R$, $i\in I\cup J$, 
are continuous for each $x\in\dom\Gamma$.
Finally, we will, at some instances, exploit the following additional
assumptions.
\begin{itemize}[leftmargin=4em]
	\item[\textbf{(A1)}]\label{ass:A1} 
		For each $x\in\R^n$, the functions $h_i(x,\cdot)\colon\R^m\to\R$, $i\in I$, 
		are convex while the functions $h_i(x,\cdot)\colon\R^m\to\R$, $i\in J$, are affine.
	\item[\textbf{(A2)}]\label{ass:A2} 
		The set-valued mapping $\Gamma$ is locally bounded at $\bar x\in\dom\Gamma$.
\end{itemize}

Subsequently, we will first derive a sequential characterization of R-regularity which holds
under validity of the aforementioned conditions. Afterwards, we will relate this sequential
characterization with the validity of the constraint qualification RCPLD. 

\subsection{A sequential characterization of R-regularity}\label{sec:sequential_R_regularity}

For some parameter $x\in\dom\Gamma$ and $\nu\notin\Gamma(x)$, $\Pi(\nu,\Gamma(x))$ equals the
solution set of
\[
	\min\limits_y\bigl\{\norm{y-\nu}\,\bigl|\,y\in\Gamma(x)\bigr\}
\]
since $\Gamma(x)$ is a closed set by continuity of $h_1(x,\cdot),\ldots,h_p(x,\cdot)$.
Due to $\nu\notin \Gamma(x)$, the objective function of the above problem is
continuously differentiable in a neighborhood of all points from $\Pi(\nu,\Gamma(x))$.
Thus, it is reasonable to investigate the associated Lagrange multiplier set
\[
	\Lambda_\nu(x,y)
	:=
	\left\{
		\lambda\in\R^p\,\middle|\,
			\frac{y-\nu}{\norm{y-\nu}}+\sum\limits_{i=1}^p\lambda_i\nabla_yh_i(x,y)=0,\,
			\forall i\in I\colon\,\lambda_i\geq 0,\,\lambda_ih_i(x,y)=0
	\right\}
\]
for each $y\in\Pi(\nu,\Gamma(x))$ as long as the pair $(x,y)$ is close
to $\{\bar x\}\times\Gamma(\bar x)$. 
For some constant $M>0$, we make use of
\[
	\Lambda_\nu^M(x,y)
	:=
	\left\{
		\lambda\in\Lambda_\nu(x,y)\,\middle|\,
		\mathsmaller\sum\nolimits_{i=1}^p\abs{\lambda_i}\leq M
	\right\}.
\]
Let us note that under validity of \hyperref[ass:A1]{\textup{\textbf{(A1)}}},
the image sets of $\Gamma$ are convex which yields that the associated
projection sets from above are actually singletons. 

Using this notation, we obtain the following technical lemma.
\begin{lemma}\label{lem:characterization_of_R-regularity}
	Let \hyperref[ass:A1]{\textup{\textbf{(A1)}}} and \hyperref[ass:A2]{\textup{\textbf{(A2)}}} hold.
	Assume that there exist a constant $M>0$ as well as sequences 
	$\{x^k\}_{k\in\N}\subset\dom\Gamma$, $\{\nu^k\}_{k\in\N}\subset\R^m$, 
	and $\{y^k\}_{k\in\N}\subset\R^m$
	satisfying $x^k\to\bar x$, $\nu^k\to\bar y$, as well as
	$\nu^k\notin\Gamma(x^k)$ and $y^k\in\Pi(\nu^k,\Gamma(x^k))$ for all $k\in\N$
	such that the set $\Lambda_{\nu^k}^M(x^k,y^k)$ is nonempty for sufficiently large $k\in\N$.
	Then we have $y^k\to\bar y$ and
	\begin{equation}\label{eq:sequential_bound}
		\dist(\nu^k,\Gamma(x^k))
		\leq
		M\,\max\bigl\{0,\max\{h_i(x^k,\nu^k)\,|\,i\in I\},
					\max\{|h_i(x^k,\nu^k)|\,|\,i\in J\}\bigr\}
	\end{equation}
	for sufficiently large $k\in\N$.
\end{lemma}
\begin{proof}
	Due to \hyperref[ass:A2]{\textup{\textbf{(A2)}}}, $\{y^k\}_{k\in\N}$ is bounded.
	Fix an arbitrary convergent subsequence $\{y^{k_s}\}_{s\in\N}$ with limit
	$\tilde y\in\R^m$. 
	By assumption, for all sufficiently large $s\in\N$, we find a multiplier
	$\lambda^{k_s}\in \Lambda^M_{\nu^{k_s}}(x^{k_s},y^{k_s})$.
	Exploiting \hyperref[ass:A1]{\textup{\textbf{(A1)}}} and the 
	definition of the set $\Lambda^M_{\nu^{k_s}}(x^{k_s},y^{k_s})$, 
	we obtain
	\begin{align*}
		\norm{ y^{k_s}-\nu^{k_s}}
		&=
		\mathsmaller\sum\nolimits_{i=1}^p
			\lambda^{k_s}_{i}\nabla_yh_i(x^{k_s},y^{k_s})^\top(\nu^{k_s}-y^{k_s})
		\\
		&\leq
		\mathsmaller\sum\nolimits_{i=1}^p
			\lambda^{k_s}_i\bigl((h_i(x^{k_s},\nu^{k_s})-h_i(x^{k_s},y^{k_s})\bigr)
		=
		\mathsmaller\sum\nolimits_{i=1}^p
			\lambda^{k_s}_i h_i(x^{k_s},\nu^{k_s})\\
		&\leq
		\mathsmaller\sum\nolimits_{i\in I}\lambda^{k_s}_i\max\{0,h_i(x^{k_s},\nu^{k_s})\}
		+
		\mathsmaller\sum\nolimits_{i\in J}|\lambda^{k_s}_i|\,\bigl|h_i(x^{k_s},\nu^{k_s})\bigr|\\
		&\leq
		M\,\max\bigl\{
				0,
				\max\{h_i(x^{k_s},\nu^{k_s})\,|\,i\in I\},
				\max\{|h_i(x^{k_s},\nu^{k_s})|\,|\,i\in J\}
				\bigr\}
	\end{align*}
	for sufficiently large $s\in\N$.
	Taking the limit $s\to\infty$ yields $\norm{\tilde y-\bar y}\leq 0$, i.e., $\tilde y=\bar y$.
	Particularly, the bounded sequence $\{y^k\}_{k\in\N}$ possesses the unique accumulation
	point $\bar y$ which must be its limit. Reprising the above arguments, we infer
	the second statement of the lemma from $\dist(\nu^k,\Gamma(x^k))=\norm{y^k-\nu^k}$.
\end{proof}

Next, we exploit \cref{lem:characterization_of_R-regularity} in order to characterize
R-regularity of $\Gamma$ under validity of 
\hyperref[ass:A1]{\textup{\textbf{(A1)}}} and \hyperref[ass:A2]{\textup{\textbf{(A2)}}}.
This result is related to \cite[Theorem~3.2]{BednarczukMinchenkoRutkowski2019} and
\cite[Theorems~2 and 3]{MinchenkoStakhovski2011b} where 
these assumptions are replaced by some a-priori 
inner semicontinuity of $\Gamma$. Here, we follow the ideas used for the proof
of \cite[Theorem~2]{MinchenkoStakhovski2011b}.
\begin{theorem}\label{thm:characterization_R_regularity_convexity}
	Let \hyperref[ass:A1]{\textup{\textbf{(A1)}}} and \hyperref[ass:A2]{\textup{\textbf{(A2)}}} hold.
	Then the following statements are equivalent.
	\begin{enumerate}
		\item[(a)] The mapping $\Gamma$ is R-regular at $(\bar x,\bar y)$ w.r.t.\ $\dom \Gamma$.
		\item[(b)] There exists a constant $M>0$ such that for each sequences 
			$\{x^k\}_{k\in\N}\subset\dom\Gamma$, $\{\nu^k\}_{k\in\N}\subset\R^m$, 
			and $\{y^k\}_{k\in\N}\subset\R^m$
			satisfying $x^k\to\bar x$, $\nu^k\to\bar y$, and
			$\nu^k\notin\Gamma(x^k)$ as well as $y^k\in\Pi(\nu^k,\Gamma(x^k))$ for all $k\in\N$,
			the set $\Lambda_{\nu^k}^M(x^k,y^k)$ is nonempty for sufficiently large $k\in\N$.
	\end{enumerate}
\end{theorem}
\begin{proof}
	We show both implications separately.\\
	$\text{(a)}\Longrightarrow\text{(b)}$: 
		Let $\Gamma$ be R-regular at $(\bar x,\bar y)$ w.r.t.\ $\dom\Gamma$.
		Then we find $\kappa>0$ and $\gamma,\delta>0$ such that
		\begin{equation}\label{eq:R_regularity_with_explicit_neighborhoods}
			\begin{aligned}
			&\forall x\in \mathbb U_\gamma(\bar x)\cap\dom \Gamma\,
			\forall y\in\mathbb U_\delta(\bar y)\colon\\
			&\qquad
			\dist(y,\Gamma(x))\leq
			\kappa\,
			\max\bigl\{	0,
					\max\{h_i(x,y)\,|\,i\in I\},
					\max\{|h_i(x,y)|\,|\,i\in J\}
				\bigr\}
			\end{aligned} 
		\end{equation}
		holds. 
		Furthermore, let $\{x^k\}_{k\in\N}\subset\dom\Gamma$, $\{\nu^k\}_{k\in\N}\subset\R^m$, and
		$\{y^k\}_{k\in\N}\subset\R^m$ be sequences which satisfy the requirements in (b).
		We first show $y^k\to\bar y$. 
		Indeed, we have
		\[
			\norm{y^k-\bar y}
			\leq
			\norm{y^k-\nu^k}+\norm{\nu^k-\bar y}
			=
			\dist(\nu^k,\Gamma(x^k))+\norm{\nu^k-\bar y},
		\]
		and the term on the right tends to zero as $k\to\infty$ by R-regularity of $\Gamma$ at
		$(\bar x,\bar y)$ and continuity of $h_1,\ldots,h_p$ at $(\bar x,\bar y)$.
		
		Fix $k\in\N$ and define mappings $\Phi_k,\Psi_k\colon\R^m\to\R$ by means of
		\[
			\forall w\in\R^m\colon\quad
			\Phi_k(w):=\norm{w-\nu^k},
			\qquad
			\Psi_k(w):=\Phi_k(w)+2\dist(w,\Gamma(x^k)).
		\]
		Observing that $\Phi_k$ is globally Lipschitz continuous with Lipschitz modulus
		$1$ while $\Gamma(x^k)$ is nonempty and closed, Clarke's principle of
		exact penalization, see \cite[Proposition~2.4.3]{Clarke1983},
		implies that $y^k$ is a global minimizer of $\Psi_k$.
		
		For sufficiently large $k\in\N$, we have $x^k\in\mathbb U_\gamma(\bar x)$ and
		$y^k\in\mathbb U_{\delta/2}(\bar y)$. Consider such $k\in\N$
		and an arbitrary vector $w\in\mathbb U_{\delta/2}(y_k)$.
		Then the above considerations and \eqref{eq:R_regularity_with_explicit_neighborhoods}
		yield the estimate
		\begin{align*}
			\Phi_k(y^k)
			&=
			\Psi_k(y^k)
			\leq
			\Psi_k(w)
			=
			\Phi_k(w)+2\,\dist(w,\Gamma(x^k))\\
			&\leq
			\Phi_k(w)+2\kappa\,
				\max\bigl\{	0,
							\max\{h_i(x^k,w)\,|\,i\in I\},
							\max\{|h_i(x^k,w)|\,|\,i\in J\}
						\bigr\}\\
			&=\max
				\left\{\Phi_k(w)+\mathsmaller\sum\nolimits_{i=1}^p\lambda_ih_i(x^k,w)
					\,\middle|\,
					\begin{aligned}
						&\forall i\in I\colon\,\lambda_i\geq 0,\,\lambda_i\min\{0,h_i(x^k,w)\}=0\\
						&\mathsmaller\sum\nolimits_{i=1}^p|\lambda_i|\leq 2\kappa
					\end{aligned}
				\right\}.
		\end{align*}
		Using the function $\mathcal L_k\colon\mathbb U_{\delta/2}(y^k)\times\R^p\to\R$ and the
		set $\widetilde\Lambda_k(w)$ given by
		\begin{align*}
			\mathcal L_k(w,\lambda)
			&:=
			\Phi_k(w)+\mathsmaller\sum\nolimits_{i=1}^p\lambda_ih_i(x^k,w)\\
			\widetilde\Lambda^k(w)
			&:=
			\left\{\lambda\in\R^p\,\middle|\,
				\mathsmaller\sum\nolimits_{i=1}^p|\lambda_i|\leq 2\kappa,\,
				\forall i\in I\colon\,\lambda_i\geq 0,\,\lambda_i\min\{0,h_i(x^k,w)\}=0
			\right\}
		\end{align*}			
		for all $w\in \mathbb U_{\delta/2}(y^k)$ and $\lambda\in\R^p$, we have
		\begin{align*}
			\forall w\in\mathbb U_{\delta/2}(y^k)\colon\quad
			\Phi_k(y^k)
			\leq
			\max\bigl\{\mathcal L_k(w,\lambda)\,\bigl|\,\lambda\in\widetilde\Lambda^k(w)\bigr\}.
		\end{align*}
		By continuity of the functions $h_i(x^k,\cdot)$, $i=1,\ldots,\ell$, 
		the inclusion $\widetilde{\Lambda}^k(w)\subset\widetilde{\Lambda}^k(y^k)$
		holds for all $w\in\mathbb U_{\delta/2}(y^k)$ close enough to $y^k$ which is why we find
		$\delta'_k\in(0,\delta/2]$ such that
		\begin{equation}\label{eq:estimate_Phi_k_y_k}
			\forall w\in\mathbb U_{\delta'_k}(y^k)\colon\quad
			\Phi_k(y^k)
			\leq
			\max\bigl\{\mathcal L_k(w,\lambda)\,\bigl|\,\lambda\in\widetilde\Lambda^k(y^k)\bigr\}.
		\end{equation}
		Defining $\mathcal Q_k\colon\mathbb U_{\delta'_k}(y_k)\to\R$ by means of
		\[
			\forall w\in\mathbb U_{\delta'_k}(y^k)\colon\quad
			\mathcal Q_k(w)
			:=
			\max\bigl\{\mathcal L_k(w,\lambda)\,\bigl|\,\lambda\in\widetilde\Lambda^k(y^k)\bigr\},
		\]
		we obtain $\Phi_k(y^k)\leq\mathcal Q_k(w)$ for all $w\in\mathbb U_{\delta'_k}(y^k)$ from
		\eqref{eq:estimate_Phi_k_y_k}. Furthermore, $\Phi_k(y^k)=\mathcal Q_k(y^k)$ holds
		which is why $y^k$ is a global minimizer of $\mathcal Q_k$.
		For sufficiently large $k\in\N$, $\mathcal L_k$ is continuously differentiable at $y^k$.
		Noting that $\widetilde\Lambda^k(y^k)$ is a compact polyhedron, $\mathcal Q_k$ is
		directionally differentiable at $y^k$, and the directional derivative 
		can be approximated from above by means of
		\[
			\forall d\in\R^m\colon\quad
			\mathcal Q'_k(y^k,d)
			\leq
			\max\bigl\{\nabla_y\mathcal L_k(y^k,\lambda)^\top d\,\bigl|\,
				\lambda\in\widetilde\Lambda^k(y^k)\bigr\}
		\]
		which follows from Danskin's theorem, see \cite[Proposition~B.25]{Bertsekas1999},
		due to validity of \hyperref[ass:A1]{\textup{\textbf{(A1)}}}.
		Recalling that $y^k$ is a global minimizer of $\mathcal Q_k$, we have
		$\mathcal Q_k'(y^k,d)\geq 0$ for all $d\in\R^m$. 
		Defining a polytope $P\subset\R^m$ by means of
		\[
			P:=\{\nabla_y\mathcal L_k(y^k,\lambda)\,|\,\lambda\in\widetilde\Lambda^k(y^k)\},
		\]
		we find $\max\{\xi^\top d\,|\,\xi\in P\}\geq 0$ for all $d\in\R^m$. 
		This yields $0\in P$. By definition of $P$, $\mathcal L_k$, and $\widetilde\Lambda^k$,
		$\Lambda^{2\kappa}_{\nu^k}(x^k,y^k)\neq\varnothing$ follows.
		Since the above arguments apply to all sufficiently large $k\in\N$, (b) holds.\\
		$\text{(b)}\Longrightarrow\text{(a)}$:
		Let (b) hold and assume that $\Gamma$ is not R-regular at $(\bar x,\bar y)$ w.r.t.\
		$\dom\Gamma$. Then we find sequences $\{x^k\}_{k\in\N}\subset\dom\Gamma$ and
		$\{\nu^k\}_{k\in\N}\subset\R^m$ such that $x^k\to\bar x$, $\nu^k\to\bar y$, and
		\begin{equation}\label{eq:not_R_regular}
			\dist(\nu^k,\Gamma(x^k))\geq
			k\,
			\max\bigl\{	0,
						\max\{h_i(x^k,\nu^k)\,|\,i\in I\},
						\max\{|h_i(x^k,\nu^k)|\,|\,i\in J\}
				\bigr\}
		\end{equation}
		as well as $\nu^k\notin\Gamma(x^k)$ hold for all $k\in\N$. 
		For each $k\in\N$, we fix $y^k\in\Pi(\nu^k,\Gamma(x^k))$.
		Due to validity of (b), the set $\Lambda^M_{\nu^k}(x^k,y^k)$ is nonempty
		for sufficiently large $k\in\N$.
		By means of 
		\hyperref[ass:A1]{\textup{\textbf{(A1)}}} and \hyperref[ass:A2]{\textup{\textbf{(A2)}}},
		\cref{lem:characterization_of_R-regularity} yields a contradiction since
		\eqref{eq:sequential_bound} and \eqref{eq:not_R_regular} are incongruous.
\end{proof}

\subsection{R-regularity under RCPLD}

In this section, we want to exploit the sequential characterization of R-regularity
obtained in \cref{thm:characterization_R_regularity_convexity} in order to
show that validity of RCPLD is a sufficient criterion for R-regularity in the presence
of \hyperref[ass:A1]{\textup{\textbf{(A1)}}} and \hyperref[ass:A2]{\textup{\textbf{(A2)}}}.
This generalizes \cite[Theorem~4.2]{BednarczukMinchenkoRutkowski2019} and \cite[Theorem~4]{MinchenkoStakhovski2011b} where a-priori inner semicontinuity of $\Gamma$ at the
reference point as well as RCRCQ were the necessary ingredients to come up with
a related result in the absence of \hyperref[ass:A1]{\textup{\textbf{(A1)}}}
and \hyperref[ass:A2]{\textup{\textbf{(A2)}}}.
\begin{theorem}\label{thm:R_regularity_and_RCPLD}
	Let \hyperref[ass:A1]{\textup{\textbf{(A1)}}} and \hyperref[ass:A2]{\textup{\textbf{(A2)}}} hold.
	Suppose that RCPLD holds at each point from $\{\bar x\}\times\Gamma(\bar x)$ w.r.t.\ $\dom\Gamma$.
	Then $\Gamma$ is R-regular at each point from $\{\bar x\}\times\Gamma(\bar x)$ w.r.t.\ $\dom\Gamma$.
\end{theorem}
\begin{proof}
	Suppose that there exists $\tilde y\in\Gamma(\bar x)$ such that $\Gamma$ is not
	R-regular at $(\bar x,\tilde y)$ w.r.t.\ $\dom\Gamma$.
	Due to \cref{thm:characterization_R_regularity_convexity}, 
	this shows that for each $\sigma\in\N$, there exist sequences 
	$\{x^k_\sigma\}_{k\in\N}\subset\dom\Gamma$,
	$\{\nu^k_\sigma\}_{k\in\N}\subset\R^m$,
	and $\{y^k_\sigma\}_{k\in\N}\subset\R^m$
	satisfying $x^k_\sigma\to\bar x$, $\nu^k_\sigma\to\tilde y$, $\nu^k_\sigma\notin\Gamma(x^k_\sigma)$
	as well as $y^k_\sigma\in\Pi(\nu^k_\sigma,\Gamma(x^k_\sigma))$ for all $k\in\N$, and
	$\Lambda^\sigma_{\nu^{k_s}_\sigma}(x^{k_s}_\sigma,y^{k_s}_\sigma)=\varnothing$ for all $s\in\N$, 
	i.e., the latter holds at least on a subsequence.
	Performing a standard diagonal sequence argument, we, thus, find sequences 
	$\{x^\sigma\}_{\sigma\in\N}\subset\dom\Gamma$,
	$\{\nu^\sigma\}_{\sigma\in\N}\subset\R^m$,
	and $\{y^\sigma\}_{\sigma\in\N}$ satisfying
	$x^\sigma\to\bar x$, $\nu^\sigma\to\tilde y$, as well as
	$\nu^\sigma\notin\Gamma(x^\sigma)$, $y^\sigma\in\Pi(\nu^\sigma,\Gamma(x^\sigma))$,
	and $\Lambda^\sigma_{\nu^\sigma}(x^\sigma,y^\sigma)=\varnothing$ for all $\sigma\in\N$.
	
	Invoking \hyperref[ass:A2]{\textup{\textbf{(A2)}}} and the continuity of
	$h_1,\ldots,h_p$ at each point from $\{\bar x\}\times\Gamma(\bar x)$, we
	obtain that for each $\varepsilon>0$, there is a $\delta>0$ such that
	$\Gamma(x)\subset\Gamma(\bar x)+\varepsilon\mathbb B$ holds for all $x\in\mathbb U_\delta(\bar x)$
	since $\Gamma$ is upper semicontinuous at $\bar x$,
	see \cite[Theorem~5.19]{RockafellarWets1998} as well.
	Thus, recalling that RCPLD is locally stable, it needs to hold at the
	points $(x^\sigma,y^\sigma)$ for sufficiently large $\sigma\in\N$.
	Exploiting the fact that RCPLD is, actually, a constraint qualification,
	this implies $\Lambda_{\nu^\sigma}(x^\sigma,y^\sigma)\neq\varnothing$.
	Since we have $\Lambda^\sigma_{\nu^\sigma}(x^\sigma,y^\sigma)=\varnothing$
	from above, we conclude that each sequence $\{\mu^\sigma\}_{\sigma\in\N}$
	with $\mu^\sigma\in\Lambda_{\nu^\sigma}(x^\sigma,y^\sigma)$ for
	all $\sigma\in\N$ satisfies $\norm{\mu^\sigma}\to\infty$ as $\sigma\to\infty$.
	Choose such a sequence. Recall that this means
	\begin{subequations}\label{eq:mu_sequence}
		\begin{align}
			\label{eq:mu_sequence_Lagrangian_derivative}
				&0=\frac{y^\sigma-\nu^\sigma}{\norm{y^\sigma-\nu^\sigma}}
				+\mathsmaller\sum\nolimits_{i=1}^p\mu^\sigma_i\nabla_yh_i(x^\sigma,y^\sigma),\\
			\label{eq:mu_sequence_complementarity_slackness_1}
				&\forall i\in I(x^\sigma,y^\sigma)\colon\,\mu^\sigma_i\geq 0,\\
			\label{eq:mu_sequence_complementarity_slackness_2}
				&\forall i\in I\setminus I(x^\sigma,y^\sigma)\colon\,\mu^\sigma_i=0
		\end{align}
	\end{subequations}
	for all $\sigma\in\N$ sufficiently large.
	
	Clearly, \hyperref[ass:A2]{\textup{\textbf{(A2)}}} guarantees that $\{y^\sigma\}_{\sigma\in\N}$
	is locally bounded and, thus, converges along a subsequence (without relabeling) to some
	$\bar y\in\Gamma(\bar x)$ by continuity of $h_1,\ldots,h_p$ at each point from
	$\{\bar x\}\times\Gamma(\bar x)$. Since RCPLD holds at $(\bar x,\bar y)$ w.r.t.\
	$\dom\Gamma$, we find a neighborhood $U$ of this point as well as an
	index set $S\subset J$ satisfying the requirements (i), (ii), and (iii) from part (d) of
	\cref{def:constraint_qualifications}. Particularly, the family $(\nabla_yh_i(x,y))_{i\in S}$
	needs to be linearly independent while the vectors
	from $(\nabla_yh_i(x,y))_{i\in J\setminus S}$ need to be linearly
	dependent on the family $(\nabla_yh_i(x,y))_{i\in S}$ for all 
	$(x,y)\in U\cap(\dom\Gamma\times\R^m)$.
	For sufficiently large $\sigma\in\N$, $(x^\sigma,y^\sigma)\in U\cap(\dom\Gamma\times\R^m)$
	holds true. The above arguments lead to the existence of $\bar\mu^\sigma_i$, $i\in J$, such that
	\begin{subequations}\label{eq:mu_sequence_modified}
		\begin{align}
			\label{eq:mu_sequence_modified_gradients}
				&\mathsmaller\sum\nolimits_{i\in J}\mu^\sigma_i\nabla_yh_i(x^\sigma,y^\sigma)
				=\mathsmaller\sum\nolimits_{i\in S}\bar\mu^\sigma_i\nabla_yh_i(x^\sigma,y^\sigma),\\
			\label{eq:mu_sequence_modified_vanishing_components}
				&\forall i\in J\setminus S\colon\quad \bar\mu^\sigma_i=0
		\end{align}
	\end{subequations}
	holds for sufficiently large $\sigma\in\N$ where, additionally, the family 
	$(\nabla_yh_i(x^\sigma,y^\sigma))_{i\in S}$ is linearly independent.
	Now, \eqref{eq:mu_sequence_modified_gradients} allows to rewrite
	\eqref{eq:mu_sequence_Lagrangian_derivative} as
	\[
		0=\frac{y^\sigma-\nu^\sigma}{\norm{y^\sigma-\nu^\sigma}}
			+\mathsmaller\sum\nolimits_{i\in I}\mu^\sigma_i\nabla_yh_i(x^\sigma,y^\sigma)
			+\mathsmaller\sum\nolimits_{i\in S}\bar\mu^\sigma_i\nabla_yh_i(x^\sigma,y^\sigma)
	\]
	for sufficiently large $\sigma\in\N$.
	Observing that there are only finitely many subsets of $I$, we
	may pass to a subsequence (without relabeling) in order to guarantee 
	$I(x^\sigma,y^\sigma)=\mathfrak I$ for all $\sigma\in\N$ and some set $\mathfrak I\subset I$.
	Now, we apply \cref{lem:Caratheodory_lemma} to the situation at hand.
	Thus, for each sufficiently large $\sigma\in\N$, we find a set $I^\sigma\subset\mathfrak I$
	as well as reals $\bar\lambda^\sigma_i$, $i\in I^\sigma\cup S$, satisfying 
	$\bar\lambda^\sigma_i>0$ for all $i\in I^\sigma$, such that the family
	$(\nabla_yh_i(x^\sigma,y^\sigma))_{i\in I^\sigma\cup S}$ is linearly independent
	while
	\[
		0=\frac{y^\sigma-\nu^\sigma}{\norm{y^\sigma-\nu^\sigma}}
			+\mathsmaller\sum\nolimits_{i\in I^\sigma\cup S}
				\bar\lambda^\sigma_i\nabla_yh_i(x^\sigma,y^\sigma)
	\]
	holds for all $\sigma\in\N$. By passing once more to a subsequence (without
	relabeling), we may ensure that $I^\sigma=\mathcal I$ holds for all $\sigma\in\N$
	and some index set $\mathcal I\subset I$. 
	Let us set $\bar\lambda^\sigma_i:=0$ for all $i\in(I\setminus\mathcal I)\cup(J\setminus S)$
	in order to rewrite the above equation as
	\begin{equation}\label{eq:lambda_sequence_Lagrangian_derivative}
		0=\frac{y^\sigma-\nu^\sigma}{\norm{y^\sigma-\nu^\sigma}}
			+\mathsmaller\sum\nolimits_{i=1}^p
				\bar\lambda^\sigma_i\nabla_yh_i(x^\sigma,y^\sigma).
	\end{equation}
	Thus, we have shown $\bar\lambda^\sigma\in\Lambda_{\nu^\sigma}(x^\sigma,y^\sigma)$.
	The above arguments show $\norm{\bar\lambda^\sigma}\to\infty$ as
	$\sigma\to\infty$. Consequently, dividing \eqref{eq:lambda_sequence_Lagrangian_derivative}
	by $\norm{\bar\lambda^\sigma}$ and taking the limit $\sigma\to\infty$, we infer
	\begin{align*}
		&0=\mathsmaller\sum\nolimits_{i=1}^p
				\bar\lambda_i\nabla_yh_i(\bar x,\bar y),\\
		&\forall i\in\mathcal I\colon\,\bar\lambda_i\geq 0,\\
		&\forall i\in (I\setminus\mathcal I)\cup(J\setminus S)\colon\,\bar\lambda_i=0
	\end{align*}
	for some non-vanishing multiplier $\bar\lambda\in\R^p$ by the assumed continuity of
	the derivatives $\nabla_yh_1,\ldots,\nabla_yh_p$ at $(\bar x,\bar y)$.
	Thus, the pair of families 
	$\bigl((\nabla_yh_i(\bar x,\bar y))_{i\in\mathcal I},(\nabla_yh_i(\bar x,\bar y))_{i\in S}\bigr)$
	is positive-linearly dependent. On the other hand, we have already shown above that
	the families $(\nabla_yh_i(x^\sigma,y^\sigma))_{i\in\mathcal I\cup S}$ are linearly
	independent. This, however, contradicts the validity of RCPLD at $(\bar x,\bar y)$
	and, thus, completes the proof.	
\end{proof}

As a consequence of the above theorem and \cref{lem:inner_semicontinuity_via_R_regularity}, 
we obtain the following corollary.
\begin{corollary}\label{cor:lower_semicontinuity_of_Gamma}
	Let the assumptions of \cref{thm:R_regularity_and_RCPLD} hold.
	Then $\Gamma$ is lower semicontinuous w.r.t.\ $\dom\Gamma$ at $\bar x$.
\end{corollary}

Inspecting the proofs of \cref{lem:characterization_of_R-regularity} as well as
\cref{thm:characterization_R_regularity_convexity,thm:R_regularity_and_RCPLD},
one can check that continuity of all involved functions w.r.t.\ the set $\dom\Gamma\times\R^m$
is actually enough in order to proceed.  
A remark, which provides another slight generalization of our setting, is presented below.
\begin{remark}\label{rem:some_more_discontinuities}
	Observe that the proofs of \cref{lem:characterization_of_R-regularity} 
	as well as 
	\cref{thm:characterization_R_regularity_convexity,thm:R_regularity_and_RCPLD} 
	remain true
	in the following setting which is slightly more general than the one
	of \cref{ass:differentiability_properties}: For each $i\in\{1,\ldots,p\}$,
	there exist functions $g_i\colon\R^n\times\R^m\to\R$ and $t_i\colon\R^n\to\overline\R$
	such that $h_i(x,y)=g_i(x,y)+t_i(x)$ holds true for all $(x,y)\in\R^n\times\R^m$.
	Furthermore, $g_i$ is continuous as well as continuously differentiable
	w.r.t.\ $y$ in a neighborhood of $\{\bar x\}\times\Gamma(\bar x)$. 
	Finally, we have $|t_i(x)|<\infty$ for all $x\in\dom\Gamma$ from a neighborhood
	of $\bar x$ and $t_i$ is continuous at $\bar x$. 
\end{remark}

Observe that the assertion of \cref{thm:R_regularity_and_RCPLD} is essentially
different from the one of \cite[Theorem~4.2]{BednarczukMinchenkoRutkowski2019}.
In \cite{BednarczukMinchenkoRutkowski2019}, the authors claimed validity of
inner semicontinuity and RCRCQ at \emph{one} point from the graph of $\Gamma$
in order to obtain R-regularity at the reference point. Here, however, we
postulate \hyperref[ass:A1]{\textup{\textbf{(A1)}}} and assume validity of RCPLD
at \emph{all} points from $\{\bar x\}\times\Gamma(\bar x)$ in order to deduce
R-regularity of $\Gamma$ at \emph{all} these points. Thus, in this setting, one
may interpret the statement of \cref{thm:R_regularity_and_RCPLD} as a
sufficient condition for lower semicontinuity of $\Gamma$ as well, see 
\cref{cor:lower_semicontinuity_of_Gamma}. Observe that we cannot modify
the statement of \cref{thm:R_regularity_and_RCPLD} in such a way that assuming
validity of RCPLD at one reference point $(\bar x,\bar y)\in\gph\Gamma$ ensures
R-regularity of $\Gamma$ at the same point without adding inner semicontinuity
of $\Gamma$ at $(\bar x,\bar y)$ while relying on the provided proof. 
However, we obtain the following result
which generalizes \cite[Theorem~4.2]{BednarczukMinchenkoRutkowski2019}.
\begin{theorem}\label{thm:R_regularity_and_RCPLD_with_inner_semicontinuity}
	Assume that $\Gamma$ is inner semicontinuous at $(\bar x,\bar y)$ w.r.t.\ $\dom\Gamma$
	and let RCPLD hold at this point w.r.t.\ $\dom\Gamma$.
	Then $\Gamma$ is R-regular at $(\bar x,\bar y)$ w.r.t.\ $\dom\Gamma$.
\end{theorem}
\begin{proof}
	We follow the lines of the proof of \cref{thm:R_regularity_and_RCPLD}
	while respecting the following changes: First, the role of
	$\tilde y$ is played by $\bar y$. Second, 
	inner semicontinuity of $\Gamma$ at $(\bar x,\bar y)$ w.r.t.\ $\dom\Gamma$ ensures
	validity of the sequential characterization of R-regularity from
	\cref{thm:characterization_R_regularity_convexity} in the absence
	of \hyperref[ass:A1]{\textup{\textbf{(A1)}}} and \hyperref[ass:A2]{\textup{\textbf{(A2)}}},
	see \cite[Theorem~3.2]{BednarczukMinchenkoRutkowski2019}.
	Third, inner semicontinuity
	of $\Gamma$ at $(\bar x,\bar y)$ w.r.t.\ $\dom\Gamma$ can be used to infer the convergence
	$y^\sigma\to\bar y$ without presuming validity of \hyperref[ass:A2]{\textup{\textbf{(A2)}}}.
	Fourth, the relation $\Lambda_{\nu^\sigma}(x^\sigma,y^\sigma)\neq\varnothing$ follows
	for sufficiently large $\sigma\in\N$ directly from local stability of RCPLD.
\end{proof}

Let us note that all the assumptions of \cref{thm:R_regularity_and_RCPLD_with_inner_semicontinuity}
hold whenever MFCQ is valid at $(\bar x,\bar y)\in\gph\Gamma$. 
In this case, $\Gamma$ is inner semicontinuous
at $(\bar x,\bar y)$, see \cite[Section~III]{Fiacco1983} for a demonstration, which particularly
yields that $\bar x$ is an interior point of $\dom\Gamma$, i.e., 
$\Gamma$ is R-regular at $(\bar x,\bar y)$ in this case.
Note that, on the other hand, validity of MFCQ at $(\bar x,\bar y)$ guarantees R-regularity of $\Gamma$
at this point by means of e.g.\ \cite[Theorem~3.2]{Borwein1986}. 
Due to \cref{lem:inner_semicontinuity_via_R_regularity}, this also shows that $\Gamma$ is inner
semicontinuous at $(\bar x,\bar y)$ and, thus, that $\bar x$ belongs to the interior of
$\dom\Gamma$.

Let us point out that in case where $\Gamma$ does not depend on the parameter $x$, \cref{thm:R_regularity_and_RCPLD_with_inner_semicontinuity} provides a sufficient condition 
for the presence of an error bound at some reference point
of a nonlinear constraint system. For a similar result under slightly stronger assumptions, 
we refer the interested reader to \cite[Theorem~7]{AndreaniHaeserSchuverdtSilva2012}.
Furthermore, we would like to mention \cite[Theorem~4.2]{ChieuLee2013} where this result
has been obtained in the context of mathematical problems with complementarity constraints.

Using the popular tools of \emph{directional} limiting variational analysis, 
the authors in \cite{GfrererMordukhovich2017} struck a completely different path
in order to derive first- and second-order sufficient
conditions for the R-regularity of $\Gamma$ which are also weaker than MFCQ.
However, in order to obtain a first-order sufficient condition in terms of initial problem data
from \cite[Theorem~3.5]{GfrererMordukhovich2017},
differentiability of the functions $h_1,\ldots,h_p$ w.r.t.\ the parameter 
as well as injectivity of the associated derivative seems to be necessary, 
and this is far beyond the regularity which was necessary in order to derive 
\cref{thm:R_regularity_and_RCPLD,thm:R_regularity_and_RCPLD_with_inner_semicontinuity}.

The upcoming example, which closes this section, shows that the statements of
\cref{thm:R_regularity_and_RCPLD,thm:R_regularity_and_RCPLD_with_inner_semicontinuity}
do not need to hold in the absence of the convexity assumption
\hyperref[ass:A1]{\textup{\textbf{(A1)}}}
or the inner semicontinuity of $\Gamma$ at the reference point, respectively.
\begin{example}\label{ex:necessity_of_assumptions}
	We consider the mapping $\Gamma\colon\R\tto\R$ given by
	\[
		\forall x\in\R\colon\quad
		\Gamma(x):=\{y\in\R\,|\,x-y\leq 0,\,y-y^2\leq 0,\,y-1\leq 0\}.
	\]
	A simple calculation reveals
	\[
		\forall x\in\R\colon\quad
		\Gamma(x)
		=
		\begin{cases}
			[x,0]\cup\{1\}	&x\in(-\infty,0],\\
			\{1\}			&x\in(0,1],\\
			\varnothing		&x\in(1,\infty).
		\end{cases}
	\]
	We study the point $\bar x:=0$ as well as the associated
	images $\bar y:=0$ and $\tilde y:=1$ in $\Gamma(\bar x)$.
	Note that $\Gamma$  is inner semicontinuous at $(\bar x,\tilde y)$
	but not at $(\bar x,\bar y)$. Thus, $\Gamma$ cannot
	be R-regular at $(\bar x,\bar y)$ due to \cref{lem:inner_semicontinuity_via_R_regularity}.
	
	Observe that the family $(-1,1-2y)$ 
	is positive-linearly dependent around $\bar y$ 
	while the family $(1-2y,1)$ is positive-linearly dependent
	around $\tilde y$. Thus, RCPLD is valid at $(\bar x,\bar y)$
	and $(\bar x,\tilde y)$, respectively.	
	This shows that the statement of \cref{thm:R_regularity_and_RCPLD} does not
	generally hold in the absence of \hyperref[ass:A1]{\textup{\textbf{(A1)}}}
	while the assertion of \cref{thm:R_regularity_and_RCPLD_with_inner_semicontinuity}
	is not generally true if $\Gamma$ is not inner semicontinuous at the
	reference point.
\end{example}

\section{Applications}\label{sec:applications}

\subsection{Parametric optimization}\label{sec:parametric_optimization}

For a function $f\colon\R^n\times\R^m\to\R$, we investigate the
parametric optimization problem
\begin{equation}\label{eq:parametric_optimization_problem}\tag{P$(x)$}
	\min\limits_y\{f(x,y)\,|\,y\in\Gamma(x)\}
\end{equation}
where $\Gamma\colon\R^n\tto\R^m$ is the set-valued mapping given in
\eqref{eq:definition_of_Gamma}. Associated with the problem \eqref{eq:parametric_optimization_problem}
are the solution mapping $S\colon\R^n\tto\R^m$ given by
\[
	\forall x\in\R^n\colon\quad
	S(x):=\argmin_y\{f(x,y)\,|\,y\in\Gamma(x)\}
\]
as well as the optimal value (or marginal) function $\varphi\colon\R^n\to\overline{\R}$ defined via
\[
	\forall x\in\R^n\colon\quad
	\varphi(x):=\inf_y\{f(x,y)\,|\,y\in\Gamma(x)\}.
\]
Clearly, we have the relation
\[
	\forall x\in\R^n\colon\quad
	S(x)=\{y\in\Gamma(x)\,|\,f(x,y)\leq\varphi(x)\}
\]
which is why $S$ can be interpreted as a solution mapping associated with a parametric
system of nonlinear inequalities and equations. 
It is well known that under comparatively weak assumptions, the optimal value function
$\varphi$ is continuous at a given reference point, see e.g.\ 
\cite{BankGuddatKlatteKummerTammer1983}.
Keeping \cref{rem:some_more_discontinuities} in mind, we are thus in position to apply the
theory from \cref{sec:sufficient_conditions_for_R_regularity} to this representation
of $S$ in order to infer its R-regularity at a given reference point under suitable
assumptions. This way, we also obtain new sufficient criteria for the presence of the
Aubin property of $S$ or its inner semicontinuity at a given reference point.
For the sake of brevity and consistency, we define $h_0\colon\R^n\times\R^m\to\overline{\R}$
by means of
\[
	\forall x\in\R^n\,\forall y\in\R^m\colon\quad
	h_0(x,y):=f(x,y)-\varphi(x)
\]
and emphasize that $S$ possesses the representation
\begin{equation}\label{eq:representation_of_S}
	\forall x\in\R^n\colon\quad
	S(x)
	=
	\left\{
		y\in\R^m\,\middle|\,
		\begin{aligned}
			h_i(x,y)&\leq 0&&i\in I\cup\{0\}\\
			h_i(x,y)&=0&&i\in J
		\end{aligned}
	\right\}.
\end{equation}
This representation of $S$ can be addressed with the theory from
\cref{sec:sufficient_conditions_for_R_regularity}.
In this section, we need to refer to the parametric constraint systems
induced by $\Gamma$ and $S$, individually. In this regard, we will exploit
the notation RCPLD$_\Gamma$ and RCPLD$_{S}$ in order to avoid
any confusion. 

Let us emphasize that, if not stated otherwise,
we will include the constraint function $h_0$ as an inequality
constraint when considering $S$, i.e., we exploit the representation
of $S$ from \eqref{eq:representation_of_S} in most of the cases.
However, it is also possible to incorporate $h_0$ as an
equality constraint.
\begin{remark}\label{rem:other representation_of_S}
	We also have the representation
	\[
		\forall x\in\R^n\colon\quad
		S(x)
		=
		\left\{
			y\in\R^m\,\middle|\,
			\begin{aligned}
				h_i(x,y)&\leq 0&&i\in I\\
				h_i(x,y)&=0&&i\in J\cup\{0\}
			\end{aligned}
		\right\},
	\]
	and, in some situations, it might be beneficial to apply the theory of 
	\cref{sec:sufficient_conditions_for_R_regularity}
	to this representation of $S$ instead of the one from
	\eqref{eq:representation_of_S}.
\end{remark}

We postulate the following standing assumption throughout the section.
\begin{assumption}\label{ass:parametric_optimization}
	The functions $f$ and $h_1,\ldots,h_p$ are continuously differentiable.
\end{assumption}
Note that by continuity of $h_1,\ldots,h_p$, we already know that $\gph\Gamma$ is
closed. Particularly, the image sets of $\Gamma$ are closed. By continuity of
$f$, we even know that the image sets of $S$ are closed.

Finally, we will exploit the following modified version of 
\hyperref[ass:A1]{\textup{\textbf{(A1)}}}
in some situations:
\begin{enumerate}
	[leftmargin=4em]
	\item[\textbf{(A1')}]\label{ass:A1mod} 
		For each $x\in\R^n$, the functions $f(x,\cdot)\colon\R^m\to\R$ and 
		$h_i(x,\cdot)\colon\R^m\to\R$, $i\in I$, 
		are convex while the functions $h_i(x,\cdot)\colon\R^m\to\R$, $i\in J$, are affine.
\end{enumerate}
We note that \hyperref[ass:A1mod]{\textup{\textbf{(A1')}}} is the counterpart of
\hyperref[ass:A1]{\textup{\textbf{(A1)}}} which addresses the representation
of $S$ from \eqref{eq:representation_of_S}. In case where one aims to exploit
the representation of $S$ from \cref{rem:other representation_of_S}, the convexity of
$f(x,\cdot)\colon\R^m\to\R$ for each $x\in\R^n$ has to be replaced by the
property of this mapping to be affine.

\subsubsection{Continuity properties of marginal functions}

In the subsequent lemma, we collect some results regarding the
continuity properties of the function $\varphi$. 
The proof is stated for the reader's convenience.
\begin{lemma}\label{lem:continuity_of_varphi}
	Fix a point $\bar x\in\dom\Gamma$ where \hyperref[ass:A2]{\textup{\textbf{(A2)}}} is
	valid. Then the following assertions hold.
	\begin{enumerate}
		\item[(a)] The function $\varphi$ is lower semicontinuous at $\bar x$.
		\item[(b)] Assume that there exists $\bar y\in\Gamma(\bar x)$ such that
			$\Gamma$ is inner semicontinuous at $(\bar x,\bar y)$ w.r.t.\ $\dom\Gamma$. 
			Then $\varphi$ is continuous at $\bar x$ w.r.t.\ $\dom\Gamma$.
		\item[(c)] Assume that $\Gamma$ possesses the Aubin property at each
			point from $\{\bar x\}\times S(\bar x)$. Then $\varphi$ is
			locally Lipschitz continuous at $\bar x$.
		\item[(d)] Assume that there exists $\bar y\in S(\bar x)$ such that
			$\Gamma$ possesses the Aubin property at $(\bar x,\bar y)$ while
			$S$ is inner semicontinuous at this point. Then $\varphi$ is locally
			Lipschitz continuous at $\bar x$.			
	\end{enumerate}
\end{lemma}
\begin{proof}
	\begin{enumerate}
		\item[(a)] By continuity of the functions $h_1,\ldots,h_p$ and validity
			of \hyperref[ass:A2]{\textup{\textbf{(A2)}}}, we obtain upper semicontinuity
			of $\Gamma$ at $\bar x$. Thus, the desired assertion can be distilled from 		
			\cite[Theorem~4.2.1]{BankGuddatKlatteKummerTammer1983} since $f$ is
			continuous.
		\item[(b)] Consulting the proof of \cite[Theorem~4.2.1]{BankGuddatKlatteKummerTammer1983},
			inner semicontinuity of $\Gamma$ at $(\bar x,\bar y)$ is enough to guarantee
			that $\varphi$ is upper semicontinuous at $\bar x$ since $f$ is continuous.
			Combining this with (a), the desired result follows.
		\item[(c)] Due to validity of \hyperref[ass:A2]{\textup{\textbf{(A2)}}}, the solution
			mapping $S$ is locally bounded at $\bar x$ as well. Particularly, $S$ possesses
			bounded images in a neighborhood of $\bar x$. Due to $\bar x\in\dom\Gamma$, we have
			$\Gamma(\bar x)\neq \varnothing$ and, thus, $S(\bar x)\neq\varnothing$ 
			by Weierstrass' theorem. Since $\Gamma$ possesses the
			Aubin property at each point from $\{\bar x\}\times S(\bar x)$, $\Gamma$ is
			inner semicontinuous at each point $(\bar x,y)\in\gph S$ and, thus, 
			possesses nonempty image sets in a
			neighborhood of $\bar x$. Thus, we deduce that $S$ possesses bounded
			and nonempty image sets in a neighborhood of $\bar x$. 
			Furthermore, $\varphi$ is lower semicontinuous at $\bar x$ by (a).
			Thus, the statement follows from \cite[Theorem~5.3(ii)]{MordukhovichNam2005}.
		\item[(d)] This follows directly from \cite[Theorem~5.3(i)]{MordukhovichNam2005}
			while observing that $\varphi$ is continuous at $\bar x$ by inner semicontinuity
			of $S$ at $(\bar x,\bar y)$ and continuity of $f$.
	\end{enumerate}
\end{proof}
We would like to mention that statement (d) of \cref{lem:continuity_of_varphi} holds even true in the
absence of \hyperref[ass:A2]{\textup{\textbf{(A2)}}} since the latter has not been used in the proof.

As a corollary of \cref{thm:R_regularity_and_RCPLD,thm:R_regularity_and_RCPLD_with_inner_semicontinuity} as well as \cref{lem:continuity_of_varphi}, we obtain the following result as a consequence of the 
local Lipschitz continuity of the functions $h_1,\ldots,h_p$ 
since the latter implies that R-regularity of $\Gamma$ at some point of its graph already
guarantees validity of the Aubin property there.

\begin{corollary}\label{cor:local_Lipschitz_continuity_of_varphi_via_RCPLD}
	Fix some point $\bar x\in\dom\Gamma$. Let one of the following additional assumptions be valid.
	\begin{enumerate}
		\item[(a)] Let \hyperref[ass:A1]{\textup{\textbf{(A1)}}} and
		 	\hyperref[ass:A2]{\textup{\textbf{(A2)}}} hold.
			Furthermore, let RCPLD$_\Gamma$ hold at each point from $\{\bar x\}\times \Gamma(\bar x)$
			and assume that $\bar x$ is an interior point of $\dom\Gamma$.
		\item[(b)] Let $\bar y\in S(\bar x)$ be chosen such that  $S$ is inner
			semicontinuous at $(\bar x,\bar y)$ while RCPLD$_\Gamma$ holds at this point.
	\end{enumerate}
	Then $\varphi$ is locally Lipschitz continuous at $\bar x$.
\end{corollary}

Let us mention that in the presence of \hyperref[ass:A1]{\textup{\textbf{(A1)}}}, 
the validity of MFCQ at
\emph{one} point $(\bar x,\bar y)\in\gph\Gamma$ implies that Slater's 
constraint qualification is
valid for the set $\Gamma(\bar x)$, 
i.e., there is some $\tilde y\in\R^m$ satisfying
$h_i(\bar x,\tilde y)<0$ for all $i\in I$ and the gradients
$(\nabla_yh_i(\bar x,\cdot))_{i\in J}$ 
(which, by validity of \hyperref[ass:A1]{\textup{\textbf{(A1)}}}, do not depend on $y$) 
are linearly independent.
The latter, however, guarantees that MFCQ and, thus,
RCPLD$_\Gamma$ hold at \emph{each} point from $\{\bar x\}\times\Gamma(\bar x)$.
As mentioned earlier, validity of MFCQ at $(\bar x,\bar y)$ also ensures that
$\bar x$ is an interior point of $\dom\Gamma$.
Thus, the regularity assumptions in the first statement of \cref{cor:local_Lipschitz_continuity_of_varphi_via_RCPLD}
are weaker than postulating validity of MFCQ at one point from $\{\bar x\}\times\Gamma(\bar x)$, and
the latter is a classical assumption in the literature to guarantee local Lipschitz
continuity of marginal functions, see e.g.\
\cite[Theorem~1]{KlatteKummer1985}.

We would like to point out that the assumption on $\bar x$ in the first statement of
\cref{cor:local_Lipschitz_continuity_of_varphi_via_RCPLD} to be an interior point
of $\dom\Gamma$ is, in general, indispensable in order to infer the local Lipschitz
continuity of $\varphi$ at this point since \cref{thm:R_regularity_and_RCPLD} only
provides R-regularity, and, thus, the Aubin property, of $\Gamma$ w.r.t.\ $\dom\Gamma$.
Observe that the assumptions of the second statement of 
\cref{cor:local_Lipschitz_continuity_of_varphi_via_RCPLD} already imply that $\bar x$
is an interior point of $\dom S$.
\begin{example}\label{ex:discontinuous_marginal_function_despite_RCPLD}
	Let us consider the simple parametric optimization problem
	\[
		\min\limits_y\{y\,|\,0\leq y\leq x\}.
	\]
	Observing that all involved functions are fully linear, RCPLD$_\Gamma$ holds at
	each point of $\gph\Gamma$ in this example. Nevertheless, the associated optimal
	value function $\varphi$ is discontinuous at $\bar x:=0$ which is a boundary
	point of $\dom\Gamma=[0,\infty)$.
	However, we note that $\varphi$ is Lipschitz continuous w.r.t.\ $\dom\Gamma$.
\end{example}

It is also possible to obtain Lipschitzian properties of the optimal value function $\varphi$
w.r.t.\ $\dom\Gamma$ without relying on the fundamentals of variational analysis,
which were used in \cite{MordukhovichNam2005}, but exploiting
the concept of R-regularity directly.
\begin{lemma}\label{lem:local:Lipschitz_continuity_of_varphi_via_R_regularity}
	Fix some point $\bar x\in\dom\Gamma$.
	Let one of the following additional assumptions be valid.
	\begin{enumerate}
		\item[(a)] Let \hyperref[ass:A2]{\textup{\textbf{(A2)}}} hold and assume that
			$\Gamma$ is R-regular at each point from $\{\bar x\}\times S(\bar x)$
			w.r.t.\ $\dom\Gamma$.
		\item[(b)] Assume that there exists $\bar y\in S(\bar x)$ such that $\Gamma$ is
			R-regular at $(\bar x,\bar y)$ w.r.t.\ $\dom\Gamma$ while $S$ is inner
			semicontinuous at this point w.r.t.\ $\dom\Gamma$. 
	\end{enumerate}
	Then $\varphi$ is locally Lipschitz continuous at $\bar x$ w.r.t.\ $\dom\Gamma$.
\end{lemma}
\begin{proof}
	\begin{enumerate}
		\item[(a)] Due to $\bar x\in\dom\Gamma$ and validity of 
			\hyperref[ass:A2]{\textup{\textbf{(A2)}}}, we indeed know $S(\bar x)\neq\varnothing$.
			Additionally, the set $S(\bar x)$ is closed, i.e., $\{\bar x\}\times S(\bar x)$ is
			compact. Thus, we can apply \cref{lem:stability_of_R_regularity} in order to
			find constants $\kappa>0$ and $\gamma>0$ as well as an open set $O\supset S(\bar x)$
			such that \eqref{eq:R_regularity} holds with $U:=\mathbb U_\gamma(\bar x)\times O$.
			 Similar as in the proof of statement (c) of \cref{lem:continuity_of_varphi}, we
			 can ensure $S(x)\neq\varnothing$ for all $x\in\mathbb U_\gamma(\bar x)\cap\dom\Gamma$
			 if only $\gamma$ is small enough. 
			 Moreover, due to \cref{lem:inner_semicontinuity_via_R_regularity}, we know that
			 $\Gamma$ is inner semicontinuous at each point from $\{\bar x\}\times S(\bar x)$ 
			 w.r.t.\ $\dom\Gamma$. Thus, we can
			 apply statement (b) of \cref{lem:continuity_of_varphi}	in order to see that
			 $\varphi$ is continuous at $\bar x$ w.r.t.\ $\dom\Gamma$. Combining this with the
			 local boundedness of $S$ and the continuity of $h_1,\ldots,h_p$, we obtain that
			 $S$ is upper semicontinuous at $\bar x$. Thus, we can even
			 choose $\gamma$ so small that $S(x)\subset O$ holds for all 
			 $x\in\mathbb U_\gamma(\bar x)\cap \dom\Gamma$. 
			 Clearly, $\Gamma$ is upper semicontinuous
			 at $\bar x$ as well which is why we find an open set $O'\supset \Gamma(\bar x)$
			 which satisfies $O'\supset O$ and $\Gamma(x)\subset O'$ for all 
			 $x\in\mathbb U_\gamma(\bar x)\cap \dom\Gamma$ if only $\gamma$ is sufficiently small.
			 By continuous differentiability of the functions $f$ and $h_1,\ldots,h_p$, these
			 functions are Lipschitz continuous on $\mathbb B_\gamma(\bar x)\times\cl O'$. 
			 Let $L_f>0$ and $L_1,\ldots,L_p>0$ be the associated Lipschitz moduli.
			 		 
			 Now, fix $x^1,x^2\in\mathbb U_\gamma(\bar x)\cap \dom\Gamma$. Then we find
			 $y^1,y^2\in O$ such that $y^1\in S(x^1)$ and $y^2\in S(x^2)$. 
			 We exploit \cite[Proposition~2.4.3]{Clarke1983} in order to see that $y^j$ is a
			 global minimizer of that map $O'\ni y\mapsto f(x^j,y)+2L_f\dist(y,\Gamma(x^j))\in\R$
			 for $j=1,2$ as well. Particularly, we obtain
			 \[
			 	\varphi(x^j)=f(x^j,y^j)\leq f(x^j,y^{3-j})+2L_f\dist(y^{3-j},\Gamma(x^j)),
			 	\qquad j=1,2.
			 \]
			 Now, we exploit \eqref{eq:R_regularity} in order to obtain
			 \begin{align*}
			 	\varphi(x^1)
			 	&\leq
			 	f(x^1,y^2)+2L_f\,\dist(y^2,\Gamma(x^1))\\
			 	&\leq
			 	f(x^2,y^2)+f(x^1,y^2)-f(x^2,y^2)\\
			 	&\qquad
			 		+2L_f\kappa\,\max	\bigl\{	0,
			 									\max\{h_i(x^1,y^2)\,|\,i\in I\},
			 									\max\{|h_i(x^1,y^2)|\,|\,i\in J\}
			 							\bigr\}\\
			 	&\leq
			 	\varphi(x^2)+f(x^1,y^2)-f(x^2,y^2)\\
			 	&\qquad	+2L_f\kappa\,\max	\bigl\{	0,
			 									\max\{h_i(x^1,y^2)-h_i(x^2,y^2)\,|\,i\in I\},\\
			 	&\qquad\qquad\qquad\qquad\qquad
			 									\max\{|h_i(x^1,y^2)-h_i(x^2,y^2)|\,|\,i\in J\}
			 							\bigr\}\\
			 	&\leq 
			 	\varphi(x^2)+L_f\norm{x^1-x^2}
			 		+2L_f\kappa\max\{L_i\,|\,i\in I\cup J\}\norm{x^1-x^2}\\
			 	&\leq
			 	\varphi(x^2)+L_f\bigl(1+2\kappa\max\{L_i\,|\,i\in I\cup J\}\bigr)\norm{x^1-x^2}.
			 \end{align*}
			 Changing the roles of the pairs $(x^1,y^1)$ and $(x^2,y^2)$ yields the local
			 Lipschitz continuity of $\varphi$ w.r.t.\ $\dom\Gamma$.
		\item[(b)] The proof can be carried out in a similar way as in (a).
			The postulated R-regularity of $\Gamma$ at $(\bar x,\bar y)$ yields
			the existence of constants $\kappa>0$ as well as $\gamma>0$ and $\delta>0$ such
			that \eqref{eq:R_regularity_with_explicit_neighborhoods} holds. By inner
			semicontinuity of $S$ at $(\bar x,\bar y)$ w.r.t.\ $\dom\Gamma$, we
			can choose $\gamma$ and $\delta$ so small such that we have
			\[
				\forall x\in\mathbb U_\gamma(\bar x)\cap\dom\Gamma\colon\quad
				\mathbb U_{\delta/2}(\bar y)\cap S(x)\neq\varnothing.
			\]
			Furthermore, we note that by continuous differentiability of 
			$f$ and $h_1,\ldots,h_p$, these functions are Lipschitz continuous on
			$\mathbb B_\gamma(\bar x)\times\mathbb B_{2\delta}(\bar y)$ with some Lipschitz
			moduli $L_f>0$ and $L_1,\ldots,L_p>0$.
			
			Now, fix $x^1,x^2\in\mathbb U_\gamma(\bar x)\cap\dom\Gamma$.
			The above arguments yield the existence of $y^1,y^2\in\mathbb U_{\delta/2}(\bar y)$
			such that $y^1\in S(x^1)$ and $y^2\in S(x^2)$ hold.
			Exploiting \cite[Proposition~2.4.3]{Clarke1983}, we find 
			\[
				\varphi(x^j)
				=
				f(x^j,y^j)
				\leq
				f(x^j,y^{3-j})+2L_f\dist(y^{3-j},\Gamma(x^j)\cap\mathbb B_{2\delta}(\bar y)),
			 	\qquad j=1,2.
			 \]
			 Due to $y^j\in\Gamma(x^j)\cap\mathbb U_{\delta/2}(\bar y)$, we even have
			 \[
			 	\dist(y^{3-j},\Gamma(x^j)\cap\mathbb B_{2\delta}(\bar y))
			 	=
			 	\dist(y^{3-j},\Gamma(x^j)),\qquad j=1,2,
			 \]
			 and, thus, the rest of the proof can be carried out as in statement (a).
	\end{enumerate}
\end{proof}

Let us briefly mention that the first statement of the above lemma
may be interpreted as an adjustment of
\cite[Theorem~5.4]{BednarczukMinchenkoRutkowski2019} whose set of assumptions is not
complete. Indeed, in the proof of this theorem, the authors exploit the presence
of R-regularity at each point from $\{\bar x\}\times S(\bar x)$ which is not covered
by the assumptions stated there.
In \cite[Theorem~4.1]{BaiYe2020}, the authors present criteria ensuring \emph{directional} Lipschitz
continuity of $\varphi$. Therefore, they impose \emph{directional} R-regularity of the mapping $\Gamma$.
In the non-directional case, their result essentially recovers 
\cref{lem:local:Lipschitz_continuity_of_varphi_via_R_regularity}
while exploiting a different boundedness assumption.

We obtain the following corollary from 
\cref{thm:R_regularity_and_RCPLD,thm:R_regularity_and_RCPLD_with_inner_semicontinuity}
as well as \cref{lem:local:Lipschitz_continuity_of_varphi_via_R_regularity}.
\begin{corollary}\label{cor:local_Lipschitz_continuity_of_varphi_via_RCPLD_restricted}
	Fix some point $\bar x\in\dom\Gamma$. Let one of the following additional assumptions be valid.
	\begin{enumerate}
		\item[(a)] Let \hyperref[ass:A1]{\textup{\textbf{(A1)}}} and
		 	\hyperref[ass:A2]{\textup{\textbf{(A2)}}} hold.
			Furthermore, let RCPLD$_\Gamma$ w.r.t.\ $\dom\Gamma$
			hold at each point from $\{\bar x\}\times \Gamma(\bar x)$.
		\item[(b)] Let $\bar y\in S(\bar x)$ be chosen such that  $S$ is inner
			semicontinuous at $(\bar x,\bar y)$ w.r.t.\ $\dom\Gamma$
			while RCPLD$_\Gamma$ w.r.t.\ $\dom\Gamma$ holds at this point.
	\end{enumerate}
	Then $\varphi$ is locally Lipschitz continuous at $\bar x$ w.r.t.\ $\dom\Gamma$.
\end{corollary}

\subsubsection{R-regularity of solution mappings}

The following theorem provides a sufficient criterion for R-regularity of the solution
mapping $S$.
\begin{theorem}\label{thm:R_regularity_of_solution_mapping}
	Fix a point $\bar x\in\dom\Gamma$. 
	Then the following assertions hold.
	\begin{enumerate}
		\item[(a)] Let \hyperref[ass:A1mod]{\textup{\textbf{(A1')}}} and
		 	\hyperref[ass:A2]{\textup{\textbf{(A2)}}} hold.
			Furthermore, let RCPLD$_{S}$ 
			hold at each point from $\{\bar x\}\times S(\bar x)$.
			Finally, let $\varphi$ be continuous at $\bar x$.
			Then $S$ is R-regular at each
			point from $\{\bar x\}\times S(\bar x)$.
				Moreover, $S$ possesses the Aubin property at all these points.
		\item[(b)] Let $\bar y\in S(\bar x)$ be chosen such that $S$ is inner
			semicontinuous at $(\bar x,\bar y)$ 
			while RCPLD$_{S}$ holds at this point.
			Then $S$ is R-regular at $(\bar x,\bar y)$.
			Moreover, $S$ possesses the Aubin property at this point.
	\end{enumerate}
\end{theorem}
\begin{proof}
	We show both statements separately.
	\begin{enumerate}
		\item[(a)] 
				Due to continuity of $\varphi$ at $\bar x$, 
				we can apply \cref{thm:R_regularity_and_RCPLD} and  \cref{rem:some_more_discontinuities}
				in order to obtain R-regularity of $S$ at all points from $\{\bar x\}\times S(\bar x)$.
				Noting that $S(\bar x)$ is nonempty by validity of 
				\hyperref[ass:A2]{\textup{\textbf{(A2)}}}, we can fix some point
				$y\in S(\bar x)$.
				From \cref{lem:inner_semicontinuity_via_R_regularity}, we infer that $S$ is
				inner semicontinuous at $(\bar x,y)$ since $\bar x$ is an interior point of
				$\dom S$ by continuity of $\varphi$ at $\bar x$. 
				Observe that validity of RCPLD$_S$ at $(\bar x,y)$ guarantees validity
				of RCPLD$_\Gamma$ at this point. 
				Now, the second statement of \cref{cor:local_Lipschitz_continuity_of_varphi_via_RCPLD}
				ensures local Lipschitz continuity of $\varphi$ at $\bar x$.
				Consequently, locally around all points from $\{\bar x\}\times S(\bar x)$,
				the variational description 
				\eqref{eq:representation_of_S} of $S$ is given by locally Lipschitz continuous
				functions. Particularly, $S$ already possesses the Aubin property at all points
				from $\{\bar x\}\times S(\bar x)$.
		\item[(b)] 
			The proof is similar to the one of the first statement.
			However, we exploit
			\cref{thm:R_regularity_and_RCPLD_with_inner_semicontinuity} to infer R-regularity
			of $S$ at $(\bar x,\bar y)$. 
	\end{enumerate}
\end{proof}

The subsequently stated examples indicate that the continuity assumption in the first statement of
the above theorem is, unluckily, indispensable in general since it may not
follow from the postulated assumptions.
\begin{example}\label{ex_no_R-regularity_despite_RCPLD_jump}
	Once more, let us investigate the parametric optimization problem from
	\cref{ex:discontinuous_marginal_function_despite_RCPLD} 
	which satisfies \hyperref[ass:A1mod]{\textup{\textbf{(A1')}}} and
	\hyperref[ass:A2]{\textup{\textbf{(A2)}}}.
	There, we have
	\[
		\forall x\in\R\colon\quad
		S(x)=
		\begin{cases}
			\varnothing	& x\in(-\infty,0),\\
			\{0\}		& x\in[0,\infty),
		\end{cases}
		\qquad
		\varphi(x)=
		\begin{cases}
			+\infty		& x\in(-\infty,0),\\
			0			& x\in[0,\infty).
		\end{cases}
	\]
	Observing that all data functions used for the modeling of the
	given parametric optimization problem are fully linear, RCPLD$_S$
	holds at each point from $\gph S$, particularly at $(\bar x,\bar y):=(0,0)$.
	However, $\varphi$ is discontinuous at $\bar x$,
	and for $x^k:=-1/k$, $k\in\N$, we obtain
	\[
		\dist(\bar y,S(x^k))
		=
		+\infty
		>
		\kappa/k
		=
		\kappa\,\max\{0,\bar y-\varphi(x^k),-\bar y,\bar y-x^k\}
	\]
	for each $\kappa>0$ and each $k\in\N$, i.e., $S$ cannot be R-regular at $(\bar x,\bar y)$.
\end{example}

\begin{example}\label{ex:no_R_regularity_despite_RCPLD}
	We consider the parametric optimization problem
	\[
		\min\limits_y\{y_1\,|\,-1\leq y_1\leq 1,\,0\leq y_2\leq 1,\,xy_1-y_2=0\}.
	\]
	We see that this problem inherently satisfies 
	\hyperref[ass:A1mod]{\textup{\textbf{(A1')}}} and
	\hyperref[ass:A2]{\textup{\textbf{(A2)}}}.
	The associated solution mapping $S$ and the associated marginal function $\varphi$
	take the following form:
	\[
		\forall x\in\R\colon\quad
		S(x)
		=
		\begin{cases}
			\{(1/x,1)\}	&x\in(-\infty,-1),\\
			\{(-1,-x)\}	&x\in[-1,0],\\
			\{(0,0)\}	&x\in(0,\infty),
		\end{cases}
		\qquad
		\varphi(x)
		=
		\begin{cases}
			1/x	&x\in(-\infty,-1),\\
			-1	&x\in[-1,0],\\
			0	&x\in(0,\infty).
		\end{cases}
	\] 
	We fix the reference points $\bar x:=0$ and $\bar y:=(-1,0)$.
	Clearly, $\varphi$ is not continuous at $\bar x$.
	
	One can check that RCPLD$_S$ is violated at $(\bar x,\bar y)$
	when using the representation of $S$ from \eqref{eq:representation_of_S}.
	However, keeping \cref{rem:other representation_of_S} in mind, we may
	also consider the representation
	\[
		\forall x\in\R\colon\quad
		S(x)=\{(y_1,y_2)\,|\,-1\leq y_1\leq 1,\,0\leq y_2\leq 1,\,xy_1-y_2=0,\,y_1-\varphi(x)=0\}
	\]
	of $S$ in order to address the proof of \cref{thm:R_regularity_of_solution_mapping}
	since this representation still possesses the necessary convex structure w.r.t.\ $y$.
	One can easily check that RCPLD holds for this mapping at $(\bar x,\bar y)$
	since the family 
	\[
		\left(\begin{pmatrix}x\\-1\end{pmatrix},\begin{pmatrix}1\\0\end{pmatrix}\right)
	\]
	associated with the equality constraints has already constant rank $2$
	in a neighborhood of $(\bar x,\bar y)$.
	However, as observed above, $\varphi$ is not continuous at $\bar x$, i.e.,
	one cannot use \cref{thm:R_regularity_and_RCPLD} and \cref{rem:some_more_discontinuities}
	in order to infer R-regularity of the solution mapping at the reference point.
\end{example}

Fix some point $\bar x\in\dom S$.
The crucial requirement in 
\cref{thm:R_regularity_of_solution_mapping}
clearly is the validity of RCPLD$_S$ at each or only some point from $\{\bar x\}\times S(\bar x)$.
As mentioned earlier, validity of MFCQ at one point from $\{\bar x\}\times\Gamma(\bar x)$
is already enough to make sure that RCPLD$_\Gamma$ holds there as well.
Let us mention that, by definition of $\varphi$, there is no $y\in S(\bar x)$
such that $h_0(\bar x,y)<0$ holds. This indicates that MFCQ generally fails to hold when applied
to the variational description \eqref{eq:representation_of_S} of $S$ which is discussed here.
Particularly, it cannot be used as a sufficient condition for RCPLD$_S$.
More details on this issue can be found in the subsequent remark.

\begin{remark}\label{rem:failure_of_MFCQ}
	Fix some point $(\bar x,\bar y)\in\gph S$.
	It is well known that this guarantees validity of the so-called \emph{Fritz--John
	conditions}, i.e., we find $\lambda_0,\lambda_1,\ldots,\lambda_p\in\R$ 
	which do not all vanish at the same time such that
	\[
		\begin{aligned}
			&\lambda_0\nabla_yh_0(\bar x,\bar y)
			+\mathsmaller\sum\nolimits_{i=1}^p\lambda_i\nabla _yh_i(\bar x,\bar y)=0,\\
			&\forall i\in I\cup\{0\}\colon\; \lambda_i\geq 0,\\
			&\forall i\in I\colon\; \lambda_i\,h_i(\bar x,\bar y)=0
		\end{aligned}
	\]
	holds, see \cite[Proposition~3.3.5]{Bertsekas1999}. 
	This, however, shows that the constraint qualification MFCQ w.r.t.\ the
	representation \eqref{eq:representation_of_S} of the mapping $S$ 
	cannot hold at $(\bar x,\bar y)$ since the pair of families
	\[		
				\left(
					\bigl(\nabla_yh_i(\bar x,\bar y)\bigr)_{i\in \{0\}\cup I(\bar x,\bar y)},
				 	\bigl(\nabla_yh_i(\bar x,\bar y)\bigr)_{i\in J}
				 \right)
	\]
	is positive-linearly dependent.
	Thus, versions of \cref{thm:R_regularity_of_solution_mapping} which exploit
	MFCQ w.r.t.\ $S$ instead of RCPLD$_{S}$ would not be reasonable at all.
	On the other hand, simple examples reveal that RCPLD$_{S}$ can hold at $(\bar x,\bar y)$,
	see \cref{ex:nonlinear_bilevel_programming_and_partial_calmness} below as well.
\end{remark}

The following lemma provides a characterization of RCPLD$_S$ via RCPLD$_\Gamma$.
\begin{lemma}\label{lem:char_RCPLD_S}
	Fix $(\bar x,\bar y)\in\gph S$.
	Then the subsequently stated conditions are equivalent.
	\begin{enumerate}
		\item[(a)] RCPLD$_S$ is valid at $(\bar x,\bar y)$.
		\item[(b)] RCPLD$_\Gamma$ is valid at $(\bar x,\bar y)$ with some 
			neighborhood $U$ of $(\bar x,\bar y)$
			and an index set $S\subset J$ according to \cref{def:constraint_qualifications}.
			Furthermore, for each $\lambda\in\Lambda(\bar x,\bar y)$
			such that the pair of families
			\[
				\left(
					\left(\nabla _yh_i(\bar x,\bar y)\right)_{i\in\{0\}\cup I_+(\bar x,\bar y,\lambda)},
					\left(\nabla _yh_i(\bar x,\bar y)\right)_{i\in S}
				\right)
			\]
			is positive-linearly dependent, 
			the family $(\nabla_yh_i(x,y))_{i\in\{0\}\cup I_+(\bar x,\bar y,\lambda)\cup S}$
			is linearly dependent for each $(x,y)\in U$.
			Above, we used
			\[
				\Lambda(\bar x,\bar y):=
				\left\{\lambda\in\R^p\,\middle|\,
				\begin{aligned}
					&\nabla_yh_0(\bar x,\bar y)
						+\mathsmaller\sum\nolimits_{i=1}^p\lambda_i\nabla_yh_i(\bar x,\bar y)=0,\\
					&\forall i\in I\colon\;\lambda_i\geq 0,\,\lambda_ih_i(\bar x,\bar y)=0
				\end{aligned}
				\right\}
			\]
			as well as
			\[
				\forall\lambda\in\Lambda(\bar x,\bar y)\colon\quad
				I_+(\bar x,\bar y,\lambda):=
				\left\{i\in I(\bar x,\bar y)\,\middle|\,\lambda_i>0\right\}.
			\]
	\end{enumerate}
\end{lemma}
\begin{proof}
	The implication (a)$\Longrightarrow$(b) is clear by definition of RCPLD$_S$.
	Thus, let us assume that the conditions in (b) hold. 
	Particularly, due to $(\bar x,\bar y)\in\gph S$ and validity of RCPLD$_\Gamma$,
	we find $\Lambda(\bar x,\bar y)\neq\varnothing$.
	Fix an arbitrary index set $\tilde K\subset \{0\}\cup I(\bar x,\bar y)$ such that
	the pair of families
	\[
		\Bigl(
			(\nabla_yh_i(\bar x,\bar y))_{i\in\tilde K},
			(\nabla_yh_i(\bar x,\bar y))_{i\in S}
		\Bigr)
	\]
	is positive-linearly dependent. In case where $\tilde K\subset I(\bar x,\bar y)$ holds,
	the vectors from the family $(\nabla_yh_i(x,y))_{i\in \tilde K\cup S}$ are linearly 
	dependent for each $(x,y)\in U$ by validity of RCPLD$_\Gamma$. 
	Thus, we assume $0\in\tilde K$.
	Then w.l.o.g.\ we find $K\subset I(\bar x,\bar y)$ with $\{0\}\cup K\subset\tilde K$ 
	as well as $\lambda_i > 0$ ($i\in \{0\}\cup K$) and
	$\lambda_i\in\R$ ($i\in S$) satisfying
	\[
		\sum\limits_{i\in\{0\}\cup K\cup S}\lambda_i\nabla_yh_i(\bar x,\bar y)=0.
	\]
	Division by $\lambda_0$ yields
	\[
		\nabla_yh_0(\bar x,\bar y)
		+\sum\limits_{i\in K\cup S}(\lambda_i/\lambda_0)\nabla_yh_i(\bar x,\bar y)
		=0.
	\]
	Defining $\tilde\lambda\in\Lambda(\bar x,\bar y)$ by
	\[
		\forall i\in\{1,\ldots,p\}\colon\quad
		\tilde\lambda_i:=
			\begin{cases}
				\lambda_i/\lambda_0	&	i\in K\cup S,\\
				0					&	\text{otherwise},
			\end{cases}
	\]
	we find $K=I_+(\bar x,\bar y,\tilde\lambda)$. 
	Thus, the family $(\nabla_yh_i(x,y))_{i\in\{0\}\cup K\cup S}$ is linearly dependent
	for each $(x,y)\in U$. 
	Due to $\{0\}\cup K\subset\tilde K$, the family $(\nabla_yh_i(x,y))_{i\in\tilde K\cup S}$
	is linearly dependent as well.
	Consequently, RCPLD$_S$ is valid at $(\bar x,\bar y)$.
\end{proof}
Whenever LICQ holds at $(\bar x,\bar y)\in\gph S$ w.r.t.\ the inequality and equality constraints
in $\Gamma$, 
the criterion from \cref{lem:char_RCPLD_S} is notably easy to check since the associated 
Lagrange multiplier in $\Lambda(\bar x,\bar y)$ is uniquely determined while RCPLD$_\Gamma$ holds
trivially.
We depict this with the aid of the subsequently stated example.
\begin{example}\label{ex:verify_RCPLD_S}
	Let us consider the parametric optimization problem
	\[
		\min\limits_y\{(y_1+1)^2+(y_2-x)^2\,|\,y_1\geq 0,\,y_2\geq 0\}.
	\]
	For later use, we set $h_1(x,y):=-y_1$ and $h_2(x,y):=-y_2$ for all $x\in\R$ and $y\in\R^2$.
	Clearly, the constraint system satisfies LICQ at each feasible point.
	We easily find $S(x)=\{(0,\max(x,0))\}$ for each $x\in\R$ as well as
	$\Lambda(x,y)=\{(2,\max(-2x,0))\}$ for each $(x,y)\in\gph S$.\\	
	Consider $\bar x\geq 0$. In this case, we find $I_+(\bar x,\bar y,\lambda)=\{1\}$ for  
	$\bar y\in S(\bar x)$ and the associated Lagrange multiplier $\lambda\in\Lambda(\bar x,\bar y)$.
	While the vectors in
	\[
		\left(\begin{pmatrix}
			2(\bar y_1+1)\\2(\bar y_2-\bar x)
		\end{pmatrix},
		\begin{pmatrix}
		-1\\0
		\end{pmatrix}
		\right)
	\]
	are positive-linearly dependent due to $\bar y_2=\bar x$, 
	a slight perturbation of $\bar x$ makes this family linearly
	independent which is why RCPLD$_S$ fails to hold at $(\bar x,\bar y)$ in this case.\\	
	Now, fix $\bar x<0$. Here, we have $I_+(\bar x,\bar y,\lambda)=\{1,2\}$ for $\bar y\in S(\bar x)$
	and the associated Lagrange multiplier $\lambda\in\Lambda(\bar x,\bar y)$.
	Noting that any strict subfamily of
	\[
		\left(\begin{pmatrix}
			2(\bar y_1+1)\\2(\bar y_2-\bar x)
		\end{pmatrix},
		\begin{pmatrix}
		-1\\0
		\end{pmatrix},
		\begin{pmatrix}
		0\\-1
		\end{pmatrix}
		\right)
	\]
	is linearly independent while any three vectors in $\R^2$ are linearly dependent, RCPLD$_S$ holds
	at $(\bar x,\bar y)$ in this case.
\end{example}

The subsequent remark comments on a way which allows a slight generalization of
\cref{thm:R_regularity_of_solution_mapping}.
\begin{remark}\label{rem:R_regularity_of_solution_mapping}
	Let $S$ be R-regular at some point $(\bar x,\bar y)\in\gph S$ w.r.t.\ $\dom S$.
	Inspecting the proof of \cite[Theorem~5.1]{BednarczukMinchenkoRutkowski2019},
	one only needs local Lipschitz continuity of all data functions at $(\bar x,\bar y)$ 
	w.r.t.\ the set $\dom S\times\R^m$ in order to infer validity of the Aubin property
	of $S$ at $(\bar x,\bar y)$ w.r.t.\ $\dom S$.
	
	Thus, the assertions of \cref{thm:R_regularity_of_solution_mapping} 
	remain true
	if all stated assumptions and assertions are stated w.r.t\ $\dom\Gamma$ since
	this is enough to ensure local coincidence of $\dom S$ and $\dom\Gamma$. 
	Particularly, relying on the respective second statement of 
	\cref{lem:continuity_of_varphi} and
	\cref{cor:local_Lipschitz_continuity_of_varphi_via_RCPLD_restricted},
	the requirement on $\varphi$ to be continuous at $\bar x$ can be removed 
	from the assumptions which need to be postulated in the 
	counterpart associated with the first statement of 
	\cref{thm:R_regularity_of_solution_mapping}.
\end{remark}

Keeping \cref{lem:inner_semicontinuity_via_R_regularity} and 
\cref{rem:R_regularity_of_solution_mapping} in mind,
the following corollary is a direct consequence of 
\cref{thm:R_regularity_of_solution_mapping}.
Indeed, this is not surprising in the light of \cref{cor:lower_semicontinuity_of_Gamma}.
\begin{corollary}\label{cor:lower_semicontinuity_of_solution_map}
	Fix a point $\bar x\in\dom\Gamma$. 
	Let \hyperref[ass:A1mod]{\textup{\textbf{(A1')}}} and
	\hyperref[ass:A2]{\textup{\textbf{(A2)}}} hold.
	Furthermore, let RCPLD$_{S}$ w.r.t.\ $\dom\Gamma$
	hold at each point from $\{\bar x\}\times S(\bar x)$.
	Then $S$ is lower semicontinuous at $\bar x$ w.r.t.\ $\dom S$.
\end{corollary}

\subsection{Bilevel optimization}\label{sec:bilevel_optimization}

Let us now consider the bilevel optimization problem
\begin{equation}\label{eq:BPP}\tag{BPP}
	\text{``}\min\limits_x\text{''}\{F(x,y)\,|\,x\in X,\,y\in S(x)\}
\end{equation}
where $F\colon\R^n\times\R^m\to\R$ is a continuously differentiable mapping, 
$X\subset\R^n$ is a closed set, and $S\colon\R^n\tto\R^m$ is the solution mapping 
associated with \eqref{eq:parametric_optimization_problem}.
The model \eqref{eq:BPP} dates back to \cite{Stackelberg1934} 
where it has been stated first
in the context of economical game theory. The quotation marks in
\eqref{eq:BPP} emphasize that this problem is not necessarily well-determined.
Indeed, whenever there is some $x\in X\cap\dom S$ where $S(x)$ is not
a singleton, then the decision maker in \eqref{eq:BPP} cannot determine the
associated objective value and, thus, classical minimization is not applicable.
In order to avoid this shortcoming, one often replaces \eqref{eq:BPP}
by its so-called \emph{optimistic} or \emph{pessimistic} version
which are given by
\[	
	\min\limits_x\{\varphi_o(x)\,|\,x\in X\}
	\qquad
	\text{and}
	\qquad
	\min\limits_x\{\varphi_p(x)\,|\,x\in X\},
\]
respectively, where the functions 
$\varphi_o,\varphi_p\colon\R^n\to\overline{\R}$
are defined as follows:
\[
	\forall x\in\R^n\colon\quad
	\varphi_o(x):=\inf\limits_y\{F(x,y)\,|\,y\in S(x)\},
	\qquad
	\varphi_p(x):=\sup\limits_y\{F(x,y)\,|\,y\in S(x)\}.
\]
This way, the optimistic and pessimistic reformulation of \eqref{eq:BPP}
reflect a cooperative behavior and a worst-case scenario between the
decision makers in \eqref{eq:BPP} and \eqref{eq:parametric_optimization_problem},
respectively. 

Due to numerous underlying applications, e.g., from finance, chemistry, or logistics, 
bilevel optimization is one of the hot topics in mathematical programming. On the
other hand, \eqref{eq:BPP} is an inherently difficult problem. Besides the above
observation that it might not be well-defined, it suffers from inherent non-convexity,
irregularity, and the implicit character of its feasible set. That is why numerous
publications dealing with the derivation of problem-tailored optimality conditions,
constraint qualifications, and solution algorithms appeared during the last three
decades. We refer the interested reader to the monographs
\cite{Bard1998,Dempe2002,DempeKalashnikovPerezValdesKalashnykova2015} for a detailed
introduction to bilevel optimization.

Let us take a look back at the optimistic and pessimistic version of
\eqref{eq:BPP} first.
Under not too restrictive assumptions, the solution mapping $S$ is upper
semicontinuous, and this property implies lower semicontinuity of $\varphi_o$,
i.e., in case where $X$ is compact, the optimistic version of \eqref{eq:BPP} is
likely to possess a global minimizer. On the other hand, in order to guarantee
lower semicontinuity of $\varphi_p$, one has to assume that $S$ is lower
semicontinuous w.r.t.\ $\dom S$. This is quite a restrictive assumption,
but our result from \cref{cor:lower_semicontinuity_of_solution_map} depicts that it 
can be valid in particular problem settings. In this regard, the subsequent theorem
follows from our aforementioned result and \cite[Theorem~5.3]{Dempe2002}.
\begin{theorem}\label{thm:existence_of_pessimistic_solutions}
	Let \hyperref[ass:A1mod]{\textup{\textbf{(A1')}}} hold.
	Furthermore, assume that $X\subset\dom \Gamma$ holds true and
	that $\Gamma$ is locally bounded at each point from $X$.
	Additionally, let RCPLD$_{S}$ w.r.t.\ $\dom\Gamma$ hold
	at each point from $\gph S\cap(X\times\R^m)$.
	Finally, let $X$ be nonempty and compact.
	Then there exists a pessimistic solution of \eqref{eq:BPP}.
\end{theorem}

The crucial requirement in the above theorem obviously is the validity of
RCPLD$_{S}$ w.r.t.\ $\dom\Gamma$ at each point from $\gph S\cap(X\times\R^m)$,
see \cref{lem:char_RCPLD_S} and the subsequent comments for some discussion.
However, let us note that this is inherent for lower level problems of type
\begin{equation}\label{eq:special_lower_level}
	\min\limits_y\bigl\{c^\top y\,|\,By\leq b(x)\bigr\}
\end{equation}
where $c\in\R^m$ and $B\in\R^{\ell\times m}$
are matrices while $b\colon\R^n\to\R^\ell$ is a continuous function. 
This means that \eqref{eq:BPP} with the special
lower level problem \eqref{eq:special_lower_level} is likely to
possess a pessimistic solution.

Observing that the optimistic and pessimistic version of \eqref{eq:BPP}
might be interpreted as a three-level decision process, the derivation of
optimality conditions via these models is quite challenging, see e.g.\
\cite{DempeMordukhovichZemkoho2012,DempeMordukhovichZemkoho2013}.
In the literature, it is a common approach to consider
\begin{equation}\label{eq:BPP_opt}\tag{BPP$'$}
	\min\limits_{x,y}\{F(x,y)\,|\,x\in X,\,y\in S(x)\}
\end{equation}
instead. This well-defined optimization problem is closely related to
the optimistic version of \eqref{eq:BPP}, 
see \cite[Proposition~6.9]{DempeMordukhovichZemkoho2012} for details.
Furthermore, by definition of the optimal value function, one can
easily check that \eqref{eq:BPP_opt} is fully equivalent to the
single-level optimization problem
\begin{equation}\label{eq:OVR}\tag{OVR}
	\min\limits_{x,y}\{F(x,y)\,|\,x\in X,\,f(x,y)-\varphi(x)\leq 0,\,y\in\Gamma(x)\}
\end{equation}
which is commonly referred to as the \emph{optimal value reformulation} or
\emph{value function transformation} of \eqref{eq:BPP_opt}.
Although this problem is still quite challenging due to the implicit character of
$\varphi$, the general non-smoothness of $\varphi$, and its inherent
irregularity, it has been exploited intensively for the derivation of necessary optimality
conditions and solution algorithms, see e.g.\
\cite{DempeDuttaMordukhovich2007,DempeFranke2015,DempeFranke2016,DempeZemkoho2013,FischerZemkohoZhou2019,MordukhovichNamPhan2012,YeZhu1995,YeZhu2010} and the references therein.
The key idea in all these papers is to use a partial penalization argument in order
to shift the crucial constraint $f(x,y)-\varphi(x)\leq 0$ from the feasible set of
\eqref{eq:OVR} to its objective function. Whenever this penalization is locally exact,
this approach is reasonable in theory and numerical practice. Following \cite{YeZhu1995},
we refer to this property as \emph{partial calmness}.
\begin{definition}\label{def:partial_calmness}
	Let $(\bar x,\bar y)\in\R^n\times\R^m$ be a locally optimal solution of
	\eqref{eq:BPP_opt}.
	We say that this program is \emph{partially calm} at $(\bar x,\bar y)$ if
	there exist a neighborhood $U$ of $(\bar x,\bar y,0)$ and some
	constant $\kappa>0$ such that we have
	$F(x,y)-F(\bar x,\bar y)+\kappa|u|\geq 0$ for each triplet
	$(x,y,u)\in U$ which satisfies
	\[
		x\in X,\quad f(x,y)-\varphi(x)\leq u,\quad y\in\Gamma(x).
	\]
\end{definition}

Indeed, \cite[Proposition~3.3]{YeZhu1995} shows that \eqref{eq:BPP_opt} is partially
calm at one of its local minimizers $(\bar x,\bar y)$ if and only if there is some
$\kappa>0$ such that $(\bar x,\bar y)$ is a local minimizer of
\[
	\min\limits_{x,y}\{F(x,y)+\tilde\kappa(f(x,y)-\varphi(x))\,|\,x\in X,\,y\in\Gamma(x)\}
\]
for each $\tilde\kappa\geq\kappa$. Noting that the latter optimization problem may satisfy standard
constraint qualifications, the presence of partial calmness indeed opens a way
to the derivation of necessary optimality conditions for \eqref{eq:BPP_opt}
since the potential non-smoothness of $\varphi$ now can be simply handled with
suitable subdifferential constructions from variational analysis.

In \cite[Section~3]{MehlitzMinchenkoZemkoho2020}, the authors provide an overview of
conditions which are sufficient for the presence of partial calmness in bilevel
optimization. Our particular interest here lies in a result which can be distilled
from \cite[Lemmas~3.2 and 3.3]{MehlitzMinchenkoZemkoho2020}.
\begin{proposition}\label{prop:partial_calmness_via_R-regularity}
	Let $(\bar x,\bar y)\in\R^n\times\R^m$ be a local minimizer of
	\eqref{eq:BPP_opt} such that $S$ is R-regular at $(\bar x,\bar y)$
	w.r.t.\ $\dom S$.
	Furthermore, assume that the sets $\dom\Gamma$ and $\dom S$ coincide
	locally around $\bar x$. Then \eqref{eq:BPP_opt} is partially calm
	at $(\bar x,\bar y)$.
\end{proposition}
We would like to point out that a related result can be found in \cite[Theorem~6.1]{BednarczukMinchenkoRutkowski2019}.

	As mentioned in \cite[Lemma~3.3]{MehlitzMinchenkoZemkoho2020}, the assumptions of
	\cref{prop:partial_calmness_via_R-regularity} actually imply that the
	point $(\bar x,\bar y)$ corresponds to a so-called (local) \emph{unifomly weak sharp minimum}
	of the parametric optimization problem \eqref{eq:parametric_optimization_problem},
	and the latter guarantees partial calmness of \eqref{eq:BPP_opt} at $(\bar x,\bar y)$,
	see \cite[Proposition~5.1]{YeZhu1995} as well.
	However, while the presence of a uniformly weak sharp minimum is generally hard
	to verify by definition, the assumptions of \cref{prop:partial_calmness_via_R-regularity}
	can be established, e.g., using the results of \cref{sec:parametric_optimization} and, thus,
	in terms of initial data.

Consequently, we may apply 
\cref{thm:R_regularity_of_solution_mapping}
as well as \cref{rem:R_regularity_of_solution_mapping} in
order to infer new sufficient conditions for the validity of partial calmness.
\begin{theorem}\label{thm:partial_calmness_via_RCPLD}
	Let $(\bar x,\bar y)\in\R^n\times\R^m$ be a local minimizer of \eqref{eq:BPP_opt}.
	Additionally, let one of the following additional conditions hold.
	\begin{enumerate}
		\item[(a)] Let \hyperref[ass:A1mod]{\textup{\textbf{(A1')}}} and
		 	\hyperref[ass:A2]{\textup{\textbf{(A2)}}} be valid.
			Furthermore, let RCPLD$_{S}$ w.r.t.\ $\dom\Gamma$
			hold at each point from $\{\bar x\}\times S(\bar x)$.
		\item[(b)] Let $S$ be inner semicontinuous at $(\bar x,\bar y)$ w.r.t.\ $\dom\Gamma$
			and let RCPLD$_{S}$ w.r.t.\ $\dom\Gamma$ hold at this point.
	\end{enumerate}
	Then \eqref{eq:BPP_opt} is partially calm at $(\bar x,\bar y)$.
\end{theorem}

As we already observed above, the crucial assumption RCPLD$_{S}$ is generally valid
for lower level problems of type \eqref{eq:special_lower_level} which is why the local
minimizers of the associated bilevel optimization problem \eqref{eq:BPP_opt} are always
partially calm. This observation already has been made in \cite[Theorem~4.1]{MehlitzMinchenkoZemkoho2020}
and \cite[Lemma~2.1]{MinchenkoBerezhov2017}. However, we would like to point out
that our result from \cref{thm:partial_calmness_via_RCPLD} may address far more general situations
as demonstrated with the aid of the subsequent example.

\begin{example}\label{ex:nonlinear_bilevel_programming_and_partial_calmness}
	Let us consider the bilevel optimization problem
	\begin{equation}\label{eq:example_BPP}
		\min\limits_{x,y}\{(x-3/4)^2+y^2\,|\,y\in S(x)\}
	\end{equation}
	where $S\colon\R\tto\R$ is the solution mapping of the parametric optimization problem
	\[
		\min\limits_y\{(x+y-2)^2\,|\,y^2-x\leq 0,\,y\geq 0\}.
	\]
	Some computations show
	\[
		\forall x\in\R\colon\quad
		S(x)=
		\begin{cases}
			\varnothing	& x\in(-\infty,0),\\
			\{\sqrt x\}	& x\in[0,1],\\
			\{2-x\}		& x\in[1,2],\\
			\{0\}		& x\in[2,\infty),
		\end{cases}
		\quad
		\varphi(x)=
		\begin{cases}
			+\infty			& x\in(-\infty,0),\\
			(x+\sqrt x-2)^2	& x\in[0,1],\\
			0				& x\in[1,2],\\
			(x-2)^2			& x\in[2,\infty).
		\end{cases}
	\]
	We observe that $S$ is a single-valued and continuous map w.r.t.\ its domain.
	Particularly, it is inner semicontinuous w.r.t.\ $\dom S$ at each point of its
	graph. Furthermore, $\dom S=\dom\Gamma$ holds.
	Using the above formula for $S$, one can easily check that \eqref{eq:example_BPP} 
	possesses the uniquely determined global minimizer $(\bar x,\bar y):=(1/4,1/2)$
	while there is another local minimizer at $(\tilde x,\tilde y):=(11/8,5/8)$.
	
	We observe that each subsystem of the family $(2(x+y-2),2y)$ possesses constant rank around
	the reference point $(\bar x,\bar y)$, and this is sufficient for the validity of RCPLD$_{S}$
	at $(\bar x,\bar y)$, i.e., \eqref{eq:example_BPP} is partially calm at this
	point by \cref{thm:partial_calmness_via_RCPLD}.
	
	Next, we consider the point $(\tilde x,\tilde y)$. Here, the set of lower level active constraints
	is empty and the gradient of the lower level objective function vanishes but, clearly, does not
	generally vanish in a neighborhood of $(\tilde x,\tilde y)$. Thus, RCPLD$_{S}$ is violated at
	$(\tilde x,\tilde y)$, i.e., we cannot employ \cref{thm:partial_calmness_via_RCPLD} in order
	to infer partial calmness of \eqref{eq:example_BPP} at $(\tilde x,\tilde y)$.
	However, one can easily check that, for each $\kappa>0$, $(\tilde x,\tilde y)$ is not a
	local minimizer of
	\[
		\min\limits_{x,y}\{(x-3/4)^2+y^2+\kappa((x+y-2)^2-\varphi(x))\,|\,y^2-x\leq 0,\,y\geq 0\}
	\]
	(note that, locally around $(\tilde x,\tilde y)$, this is a convex problem) which is why
	\eqref{eq:example_BPP} is actually not partially calm at $(\tilde x,\tilde y)$.
\end{example}

\section{Conclusions}\label{sec:conclusions}

In this manuscript, we have shown that the validity of the constraint qualification RCPLD is
sufficient to infer the presence of R-regularity for set-valued mappings of 
type \eqref{eq:definition_of_Gamma}. Our results generalize similar considerations which
exploit the constraint qualifications MFCQ or RCRCQ for that purpose, see
\cite{BednarczukMinchenkoRutkowski2019,LudererMinchenkoSatsura2002,MinchenkoStakhovski2011b}.
We applied our findings in order to study nonlinear parametric optimization problems
and bilevel optimization problems.
First, we inferred new criteria ensuring Lipschitz continuity of optimal value functions
as well as R-regularity and lower semicontinuity of solution mappings in parametric
programming. As we have seen, a similar analysis w.r.t.\ the solution mapping is not
possible under MFCQ.
Second, these results were exploited in order to state a criterion for
the existence of solutions in pessimistic bilevel optimization as well as a sufficient
condition for the validity of the partial calmness property in optimistic bilevel
optimization. 
Throughout the manuscript, simple examples illustrated applicability but also the limits of our
findings.



\begin{thebibliography}{47}
\providecommand{\natexlab}[1]{#1}
\providecommand{\url}[1]{\texttt{#1}}
\expandafter\ifx\csname urlstyle\endcsname\relax
  \providecommand{\doi}[1]{doi: #1}\else
  \providecommand{\doi}{doi: \begingroup \urlstyle{rm}\Url}\fi

\bibitem[Andreani et~al.(2005)Andreani, Martinez, and
  Schuverdt]{AndreaniMartinezSchuverdt2005}
R.~Andreani, J.~M. Martinez, and M.~L. Schuverdt.
\newblock On the relation between constant positive linear dependence condition
  and quasinormality constraint qualification.
\newblock \emph{Journal of Optimization Theory and Applications}, 125\penalty0
  (2):\penalty0 473--483, 2005.
\newblock \doi{10.1007/s10957-004-1861-9}.

\bibitem[Andreani et~al.(2012)Andreani, Haeser, Schuverdt, and
  Silva]{AndreaniHaeserSchuverdtSilva2012}
R.~Andreani, G.~Haeser, M.~L. Schuverdt, and P.~J.~S. Silva.
\newblock A relaxed constant positive linear dependence constraint
  qualification and applications.
\newblock \emph{Mathematical Programming}, 135\penalty0 (1):\penalty0 255--273,
  2012.
\newblock \doi{10.1007/s10107-011-0456-0}.

\bibitem[Bai and Ye(2020)]{BaiYe2020}
K.~Bai and J.~J. Ye.
\newblock Directional necessary optimality conditions for bilevel programs.
\newblock \emph{preprint arXiv}, pages 1--34, 2020.
\newblock URL \url{https://arxiv.org/abs/2004.01783}.

\bibitem[Bank et~al.(1983)Bank, Guddat, Klatte, Kummer, and
  Tammer]{BankGuddatKlatteKummerTammer1983}
B.~Bank, J.~Guddat, D.~Klatte, B.~Kummer, and K.~Tammer.
\newblock \emph{Nonlinear Parametric Optimization}.
\newblock Birkh{\"a}user, Basel, 1983.

\bibitem[Bard(1998)]{Bard1998}
J.~F. Bard.
\newblock \emph{Practical {B}ilevel {O}ptimization: {A}lgorithms and
  {A}pplications}.
\newblock Kluwer Academic, Dordrecht, 1998.

\bibitem[Bednarczuk et~al.(2020)Bednarczuk, Minchenko, and
  Rutkowski]{BednarczukMinchenkoRutkowski2019}
E.~M. Bednarczuk, L.~I. Minchenko, and K.~E. Rutkowski.
\newblock On {L}ipschitz-like continuity of a class of set-valued mappings.
\newblock \emph{Optimization}, 69\penalty0 (12):\penalty0 2535--2549, 2020.
\newblock \doi{10.1080/02331934.2019.1696339}.

\bibitem[Bertsekas(1999)]{Bertsekas1999}
D.~P. Bertsekas.
\newblock \emph{Nonlinear Programming}.
\newblock Athena Scientific, Belmot, 1999.

\bibitem[Borwein(1986)]{Borwein1986}
J.~M. Borwein.
\newblock Stability and regular points of inequality systems.
\newblock \emph{Journal of Optimization Theory and Applications}, 48\penalty0
  (1):\penalty0 9--52, 1986.
\newblock \doi{10.5555/3182697.3183275}.

\bibitem[Bosch et~al.(2004)Bosch, Jourani, and
  Henrion]{BoschJouraniHenrion2004}
P.~Bosch, A.~Jourani, and R.~Henrion.
\newblock Sufficient conditions for error bounds and applications.
\newblock \emph{Applied Mathematics and Optimization}, 50\penalty0
  (2):\penalty0 161--181, 2004.
\newblock \doi{10.1007/s00245-004-0799-5}.

\bibitem[Chieu and Lee(2013)]{ChieuLee2013}
N.~H. Chieu and G.~M. Lee.
\newblock A relaxed constant positive linear dependence constraint
  qualification for mathematical programs with equilibrium constraints.
\newblock \emph{Journal of Optimization Theory and Applications}, 158\penalty0
  (1):\penalty0 11--32, 2013.
\newblock \doi{10.1007/s10957-012-0227-y}.

\bibitem[Clarke(1983)]{Clarke1983}
F.~H. Clarke.
\newblock \emph{Optimization and {N}onsmooth {A}nalysis}.
\newblock Wiley, New York, 1983.

\bibitem[Dempe(2002)]{Dempe2002}
S.~Dempe.
\newblock \emph{Foundations of Bilevel Programming}.
\newblock Kluwer, Dordrecht, 2002.

\bibitem[Dempe and Franke(2015)]{DempeFranke2015}
S.~Dempe and S.~Franke.
\newblock The bilevel road pricing problem.
\newblock \emph{International Journal of Computing and Optimization},
  2\penalty0 (2):\penalty0 71--92, 2015.
\newblock \doi{10.12988/ijco.2015.5415}.

\bibitem[Dempe and Franke(2016)]{DempeFranke2016}
S.~Dempe and S.~Franke.
\newblock On the solution of convex bilevel optimization problems.
\newblock \emph{Computational Optimization and Applications}, 63\penalty0
  (3):\penalty0 685--703, 2016.
\newblock \doi{10.1007/s10589-015-9795-8}.

\bibitem[Dempe and Zemkoho(2013)]{DempeZemkoho2013}
S.~Dempe and A.~B. Zemkoho.
\newblock The bilevel programming problem: reformulations, constraint
  qualifications and optimality conditions.
\newblock \emph{Mathematical Programming}, 138\penalty0 (1):\penalty0 447--473,
  2013.
\newblock \doi{10.1007/s10107-011-0508-5}.

\bibitem[Dempe et~al.(2007)Dempe, Dutta, and
  Mordukhovich]{DempeDuttaMordukhovich2007}
S.~Dempe, J.~Dutta, and B.~S. Mordukhovich.
\newblock New necessary optimality conditions in optimistic bilevel
  programming.
\newblock \emph{Optimization}, 56\penalty0 (5-6):\penalty0 577--604, 2007.
\newblock \doi{10.1080/02331930701617551}.

\bibitem[Dempe et~al.(2012)Dempe, Mordukhovich, and
  Zemkoho]{DempeMordukhovichZemkoho2012}
S.~Dempe, B.~S. Mordukhovich, and A.~B. Zemkoho.
\newblock Sensitivity analysis for two-level value functions with applications
  to bilevel programming.
\newblock \emph{SIAM Journal on Optimization}, 22\penalty0 (4):\penalty0
  1309--1343, 2012.
\newblock \doi{10.1137/110845197}.

\bibitem[Dempe et~al.(2014)Dempe, Mordukhovich, and
  Zemkoho]{DempeMordukhovichZemkoho2013}
S.~Dempe, B.~S. Mordukhovich, and A.~B. Zemkoho.
\newblock Necessary optimality conditions in pessimistic bilevel programming.
\newblock \emph{Optimization}, 63\penalty0 (4):\penalty0 505--533, 2014.
\newblock \doi{10.1080/02331934.2012.696641}.

\bibitem[Dempe et~al.(2015)Dempe, Kalashnikov, P{\'{e}}rez-Vald{\'{e}}z, and
  Kalashnykova]{DempeKalashnikovPerezValdesKalashnykova2015}
S.~Dempe, V.~Kalashnikov, G.~P{\'{e}}rez-Vald{\'{e}}z, and N.~Kalashnykova.
\newblock \emph{Bilevel Programming Problems - Theory, Algorithms and
  Applications to Energy Networks}.
\newblock Springer, Berlin, 2015.

\bibitem[Fabian et~al.(2010)Fabian, Henrion, Kruger, and
  Outrata]{FabianHenrionKrugerOutrata2010}
M.~J. Fabian, R.~Henrion, A.~Y. Kruger, and J.~V. Outrata.
\newblock Error bounds: necessary and sufficient conditions.
\newblock \emph{Set-Valued and Variational Analysis}, 18\penalty0 (2):\penalty0
  121--149, 2010.
\newblock \doi{10.1007/s11228-010-0133-0}.

\bibitem[Fedorov(1979)]{Fedorov1979}
V.~V. Fedorov.
\newblock \emph{Numerical maximin methods}.
\newblock Nauka, Moscow, 1979.

\bibitem[Fiacco(1983)]{Fiacco1983}
A.~V. Fiacco.
\newblock Optimal value continuity and differential stability bounds under the
  {M}angasarian--{F}romovitz constraint qualification.
\newblock In A.~V. Fiacco, editor, \emph{Mathematical Programming with Data
  Perturbations}, volume~2, pages 65--90. Marcel Dekker, New York, 1983.

\bibitem[Fischer et~al.(2019)Fischer, Zemkoho, and
  Zhou]{FischerZemkohoZhou2019}
A.~Fischer, A.~B. Zemkoho, and S.~Zhou.
\newblock Semismooth {N}ewton-type method for bilevel optimization: global
  convergence and extensive numerical experiments.
\newblock \emph{preprint arXiv}, pages 1--27, 2019.
\newblock URL \url{https://arxiv.org/abs/1912.07079}.

\bibitem[Gfrerer and Mordukhovich(2017)]{GfrererMordukhovich2017}
H.~Gfrerer and B.~S. Mordukhovich.
\newblock Robinson stability of parametric constraint systems via variational
  analysis.
\newblock \emph{SIAM Journal on Optimization}, 27\penalty0 (1):\penalty0
  438--465, 2017.
\newblock \doi{10.1137/16M1086881}.

\bibitem[Gfrerer and Outrata(2016)]{GfrererOutrata2016}
H.~Gfrerer and J.~V. Outrata.
\newblock On {L}ipschitzian properties of implicit multifunctions.
\newblock \emph{SIAM Journal on Optimization}, 26\penalty0 (4):\penalty0
  2160--2189, 2016.
\newblock \doi{10.1137/15M1052299}.

\bibitem[Guo and Lin(2013)]{GuoLin2013}
L.~Guo and G.-H. Lin.
\newblock Notes on some constraint qualifications for mathematical programs
  with equilibrium constraints.
\newblock \emph{Journal of Optimization Theory and Applications}, 156:\penalty0
  600--616, 2013.
\newblock \doi{10.1007/s10957-012-0084-8}.

\bibitem[Ioffe(1979)]{Ioffe1979}
A.~D. Ioffe.
\newblock Regular points of {L}ipschitz functions.
\newblock \emph{Transactions of the American Mathematical Society},
  251:\penalty0 61--69, 1979.
\newblock \doi{10.1090/S0002-9947-1979-0531969-6}.

\bibitem[Ioffe(2000)]{Ioffe2000}
A.~D. Ioffe.
\newblock Metric regularity and subdifferential calculus.
\newblock \emph{Russian Mathematical Surveys}, 55\penalty0 (3):\penalty0
  501--558, 2000.
\newblock \doi{10.1070/RM2000v055n03ABEH000292}.

\bibitem[Janin(1984)]{Janin1984}
R.~Janin.
\newblock Directional derivative of the marginal function in nonlinear
  programming.
\newblock In A.~V. Fiacco, editor, \emph{Sensitivity, Stability and Parametric
  Analysis}, volume~21, pages 110--126. Springer, Berlin, 1984.
\newblock \doi{10.1007/BFb0121214}.

\bibitem[Klatte and Kummer(1985)]{KlatteKummer1985}
D.~Klatte and B.~Kummer.
\newblock Stability properties of infima and optimal solutions of parametric
  optimization problems.
\newblock In V.~F. Demyanov and D.~Pallaschke, editors, \emph{Nondifferentiable
  Optimization: Motivations and Applications}, pages 215--229. Springer,
  Berlin, 1985.

\bibitem[Luderer et~al.(2002)Luderer, Minchenko, and
  Satsura]{LudererMinchenkoSatsura2002}
B.~Luderer, L.~I. Minchenko, and T.~Satsura.
\newblock \emph{Multivalued Analysis and Nonlinear Programming Problems with
  Perturbations}.
\newblock Springer Science$+$Business Media, Dordrecht, 2002.

\bibitem[Mehlitz et~al.(2020)Mehlitz, Minchenko, and
  Zemkoho]{MehlitzMinchenkoZemkoho2020}
P.~Mehlitz, L.~I. Minchenko, and A.~B. Zemkoho.
\newblock A note on partial calmness for bilevel optimization problems with
  linearly structured lower level.
\newblock \emph{Optimization Letters}, pages 1--15, 2020.
\newblock \doi{10.1007/s11590-020-01636-6}.

\bibitem[Minchenko and Berezhnov(2017)]{MinchenkoBerezhov2017}
L.~I. Minchenko and D.~E. Berezhnov.
\newblock On global partial calmness for bilevel programming problems with
  linear lower-level problem.
\newblock In \emph{CEUR Workshop Proceedings}, volume 1987, 2017.
\newblock URL \url{http://ceur-ws.org/Vol-1987/paper60.pdf}.

\bibitem[Minchenko and Stakhovski(2011{\natexlab{a}})]{MinchenkoStakhovski2011}
L.~I. Minchenko and S.~Stakhovski.
\newblock On relaxed constant rank regularity condition in mathematical
  programming.
\newblock \emph{Optimization}, 60\penalty0 (4):\penalty0 429--440,
  2011{\natexlab{a}}.
\newblock \doi{10.1080/02331930902971377}.

\bibitem[Minchenko and
  Stakhovski(2011{\natexlab{b}})]{MinchenkoStakhovski2011b}
L.~I. Minchenko and S.~Stakhovski.
\newblock Parametric nonlinear programming problems under the relaxed constant
  rank condition.
\newblock \emph{SIAM Journal on Optimization}, 21\penalty0 (1):\penalty0
  314--332, 2011{\natexlab{b}}.
\newblock \doi{10.1137/090761318}.

\bibitem[Mordukhovich(2006)]{Mordukhovich2006}
B.~S. Mordukhovich.
\newblock \emph{Variational Analysis and Generalized Differentiation. {I}:
  Basic Theory}.
\newblock Springer, Berlin, 2006.

\bibitem[Mordukhovich and Nam(2005)]{MordukhovichNam2005}
B.~S. Mordukhovich and N.~M. Nam.
\newblock Variational stability and marginal functions via generalized
  differentiation.
\newblock \emph{Mathematics of Operations Research}, 30\penalty0 (4):\penalty0
  800--816, 2005.
\newblock \doi{10.1287/moor.1050.0147}.

\bibitem[Mordukhovich et~al.(2012)Mordukhovich, Nam, and
  Phan]{MordukhovichNamPhan2012}
B.~S. Mordukhovich, N.~M. Nam, and H.~M. Phan.
\newblock Variational analysis of marginal functions with applications to
  bilevel programming.
\newblock \emph{Journal of Optimization Theory and Applications}, 152\penalty0
  (3):\penalty0 557--586, 2012.
\newblock \doi{10.1007/s10957-011-9940-1}.

\bibitem[Qi and Wei(2000)]{QiWei2000}
L.~Qi and Z.~Wei.
\newblock On the constant positive linear dependence condition and its
  application to {S}{Q}{P} methods.
\newblock \emph{SIAM Journal on Optimization}, 10\penalty0 (4):\penalty0
  963--981, 2000.
\newblock \doi{10.1137/S1052623497326629}.

\bibitem[Robinson(1976)]{Robinson1976}
S.~M. Robinson.
\newblock Stability theory for systems of inequalities, part {II}:
  differentiable nonlinear systems.
\newblock \emph{SIAM Journal of Numerical Analysis}, 13\penalty0 (4):\penalty0
  497--513, 1976.

\bibitem[Rockafellar and Wets(1998)]{RockafellarWets1998}
R.~T. Rockafellar and R.~J.-B. Wets.
\newblock \emph{Variational Analysis}, volume 317 of \emph{Grundlehren der
  mathematischen Wissenschaften}.
\newblock Springer, Berlin, 1998.

\bibitem[v.\ Stackelberg(1934)]{Stackelberg1934}
H.\ v.\ Stackelberg.
\newblock \emph{Marktform und Gleichgewicht}.
\newblock Springer, Berlin, 1934.

\bibitem[Xu and Ye(2020)]{XuYe2019}
M.~Xu and J.~J. Ye.
\newblock Relaxed constant positive linear dependence constraint qualification
  and its application to bilevel programs.
\newblock \emph{Journal of Global Optimization}, 78\penalty0 (1):\penalty0
  181--205, 2020.
\newblock \doi{10.1007/s10898-020-00907-x}.

\bibitem[Ye(1998)]{Ye1998}
J.~J. Ye.
\newblock New uniform parametric error bounds.
\newblock \emph{Journal of Optimization Theory and Applications}, 98:\penalty0
  197--219, 1998.
\newblock \doi{10.1023/A:1022649217032}.

\bibitem[Ye and Zhu(1995)]{YeZhu1995}
J.~J. Ye and D.~L. Zhu.
\newblock Optimality conditions for bilevel programming problems.
\newblock \emph{Optimization}, 33\penalty0 (1):\penalty0 9--27, 1995.
\newblock \doi{10.1080/02331939508844060}.

\bibitem[Ye and Zhu(2010)]{YeZhu2010}
J.~J. Ye and D.~L. Zhu.
\newblock New necessary optimality conditions for bilevel programs by combining
  the {M}{P}{E}{C} and value function approaches.
\newblock \emph{SIAM Journal on Optimization}, 20\penalty0 (4):\penalty0
  1885--1905, 2010.
\newblock \doi{10.1137/080725088}.

\bibitem[Yen(1997)]{Yen1997}
N.~D. Yen.
\newblock Stability of the solution set of perturbed nonsmooth inequality
  systems and applications.
\newblock \emph{Journal of Optimization Theory and Applications}, 93:\penalty0
  199--225, 1997.
\newblock \doi{10.1023/A:1022662120550}.

\end{thebibliography}

\end{document}